\documentclass{amsart}
\usepackage{fancyhdr}
\usepackage{amsmath,amssymb,amsfonts}
\usepackage{graphicx}
\usepackage{epstopdf}
\usepackage{array}
\usepackage{subcaption}
\usepackage{colonequals}
\usepackage{bbm}
\usepackage{multirow}
\usepackage{booktabs}
\usepackage{mathrsfs}
\usepackage{enumitem}
\usepackage[capitalize]{cleveref}
\usepackage{siunitx}
\usepackage{url}

\usepackage[numbers]{natbib}

\def\X{\CMcal{\X}}
\def\F{\CMcal{F}}
\def\X{\CMcal{X}}

\def\F{\mathcal{F}}

\renewcommand{\phi}{\varphi}
\renewcommand{\epsilon}{\varepsilon}
\renewcommand{\star}{\cdot}
\newcommand{\cpp}{\mathrm{cpp}}
\newcommand{\nototimes}{{\,\bar{\otimes}\,}}
\newcommand{\leb}{{\boldsymbol \lambda}}

\newtheorem{problem}{Variational Problem}
\newtheorem{my_algorithm}{Algorithm}

\newtheorem{definition}{Definition}
\newtheorem{proposition}{Proposition}
\newtheorem*{proposition*}{Proposition}
\newtheorem{appendix-proposition}{Appendix Proposition}

\newcommand{\bmhead}{\section*}

\theoremstyle{remark}
\newtheorem{remark}{Remark}

\graphicspath{{.},{./Figures/}}

\begin{document}
\bibliographystyle{./plainnat_modified}

\title [Image Varifolds on Meshes]{Image Varifolds on Meshes for Mapping Spatial Transcriptomics}
    
\author{Michael I. Miller}
\address{MM: Center of Imaging Science and 
    Department of Biomedical Engineering, Johns Hopkins
    University}
\email{mim@jhu.edu}

\author{Alain Trouv\'e}
\address{AT: Centre Giovanni Borelli (UMR 9010),Ecole Normale Supérieure Paris-Saclay, Université Paris-Saclay}
\email{alain.trouve@cmla.ens-cachan.fr}         

\author{Laurent Younes}
\address{LY: Center of imaging Science and
    Department of Applied Mathematics and Statistics, Johns Hopkins
    University}
\email{laurent.Younes@jhu.edu}

    \date{\today}
\begin{abstract}
 Advances in the development of largely automated microscopy methods such as MERFISH for imaging cellular structures in mouse brains  are providing spatial detection of micron resolution gene expression.
While there has been tremendous progress made in the field Computational Anatomy (CA) 
to perform diffeomorphic mapping technologies at the tissue scales
 for advanced neuroinformatic studies in common coordinates,
integration of molecular- and cellular-scale populations through statistical averaging via common coordinates remains yet unattained. 
This paper describes the first set of algorithms for calculating geodesics in the space of diffeomorphisms, what we term
Image-Varifold LDDMM,
extending the family of large deformation diffeomorphic metric mapping (LDDMM) algorithms
to accommodate
the ``copy and paste'' varifold action of particles 
which extends consistently to the tissue scales.

We represent the brain data as geometric measures, termed as {\em image varifolds}
supported by a large number of unstructured points, 
each point representing a small volume in space 
and carrying a list of densities of {\em features} elements of a high-dimensional feature space.
The shape of image varifold brain spaces is measured by transforming them by diffeomorphisms.
The metric between image varifolds is obtained after embedding these objects in a linear space equipped with the norm, yielding a so-called ``chordal metric.'' 
\end{abstract}


\maketitle

\section{Introduction}

We are seeing a new period of method-driven renaissance in neuroanatomy, one that is distinguished by a focus on large-scale projects generating unprecedented amounts of spatially resolved brain data across multiple, complementary modalities.  Recent years have seen many advances in the development of largely automated microscopy instruments for imaging cellular structures in mouse brain anatomy and function \cite{RN1,RN2,RN3,RN7}, including 
morphological reconstructions at the dense 3D electron microscopy (EM) \cite{RN22} and at the mesoscale by whole-brain reconstructions (as exemplified by the BRAIN Cell Census Network (BICCN) project \cite{RN9}), neuronal projectomes \cite{RN23,RN24} and brain-wide maps of cell type distributions \cite{RN8}, spatial transcriptomics technologies, such as MERFISH \cite{RN25,RN26,RN27}, STARmap \cite{RN10,RN11} generating massive amounts of gene expression data of thousands of genes at a time, as well as emergent barcoding technologies linking data on neuronal projectomes with dense transcriptional profiling at the single-cell level \cite{RN28,RN29,RN30}.

 Since the publication of the Allen ISH atlas in 2006, methods for the spatial detection of gene expression have rapidly improved in their both spatial resolution and the number of genes that can be detected simultaneously. A host of different methods, including MERFISH, STARmap, seqFISH \cite{RN31}, and others now allows for the simultaneous measurement of a few hundreds to thousands of genes, and potentially all genes, at single molecule and hence single-cell resolution, and at the scale of whole brain sections. Spatially resolved transcriptomic data further offers an opportunity for obtaining multi-modal measurements at single-cell resolution allowing for the combination of dense spatial transcriptomics with the simultaneous measurement of single-neuron projection information using barcode sequencing in BARseq2 \cite{RN28,RN29}. Spatial transcriptomics data can also be registered to functional Ca2+ imaging data, linking gene expression and neuronal activity in behaving mice \cite{RN33,RN34,RN35}. 
The importance of these technological advancements for understanding the dense metric structure of the brain by building up coarse physiological atlas scales built up from dense imaging measurements at the cellular scales was recently recognized when spatial transcriptomics was selected as Nature method of the year \cite{RN36}.

While disparate datasets are being collected from comparable brains and thus exist in a common underlying coordinate system, differences in data modalities and imaging technologies however map them into disparate spaces that need to be mapped to each other to allow integration and maximal impact of these datasets obtained under high expense. This is one of the principal goals of this paper to provide image varifold (IV) LDMMM building on the progress made in the Computational Anatomy (CA) 
\cite{GrenanderMiller1998,Ashburner2009,Pennec2011} community in the suite of methods called large deformation diffeomorphic metric mapping (LDDMM).
Advances in CA for diffeomorphic mapping technologies to atlas coordinates at the tissue scales
\cite{RN15,RN16,RN7,RN18} provide mapping technologies for advanced neuroinformatic studies in common coordinates.
Integration of molecular scale populations through statistical averaging via common coordinates remains yet unattained at the molecular and cellular scales. 
Only recently have theories been put forward that extend the diffeomorphism atlasing technologies of CA \cite{RN20,RN21,GrenanderMiller1998,Ashburner2009,Pennec2011,RN15,RN16,RN7,RN18} to the molecular scales consistently with  
diffeomorphic mapping at the tissue scales, providing mapping technologies for advanced neuroinformatic studies in common coordinates,
integration of molecular- and cellular-scale populations through statistical averaging via common coordinates.
These theories, 
introduced in \citet{RN21,miller-tward-trouve-2022},
describe the geodesic equations for building correspondences between MERFISH samples using varifold norms. This paper describes the new family of algorithms for calculating these geodesics, what we term
Image-Varifold LDDMM, extending the family of large deformation diffeomorphic metric mapping (LDDMM) algorithms \cite{RN15,Miller2006,Trouve2010a,Vialard2012,Miller2015} to accommodate
the copy-and-paste varifold action of particles described in \cite{RN21} which extends consistently to the tissue scales.

We expect that the Image-Varifold LDDMM technologies will also be important for problems emerging now in digital pathology linking the molecular scales of histology with the tissue scales of MRI for understanding neurodegenerative diseases associated to the validation and further development of biomarkers as surrogates of molecular disease, as in Alzheimer’s Disease \cite{RN18}. 

The major contribution of this paper is to adapt advanced computational methods developed in the field of Computational Anatomy for neuroanatomical data analysis by a broad neuroscience research community using spatial transcriptomics datasets. We believe IV-LDDMM can potentially be broadly used by the neuroscience community since we will describe algorithms for registering datasets collected from multiple animals as 3D stacks or as single 2D sections as well as build correspondence to the Common Coordinate Framework (CCF) template \cite{RN40} that enables standardized comparisons across different datasets and users. Further, the user will be able to overlay and segment the registered data with labels from atlases transporting the labels from the coarse tissue scale to the fine molecular scale using the diffeomorphic properties of the maps, including the Allen ARA 2008 atlas \cite{RN41}, CCFv3 2017 atlas \cite{RN42} and the recently built unified Paxinos and CCFv3 atlas \cite{RN43}. This will allow users will to be able to perform statistical analysis on multiple sections within and across animals to extract a number of useful statistics of their data leveraging the anatomical labels and  and the user will be able to share datasets within registered coordinate systems with other groups collaborating on data analysis with different laboratories.

\section{Image varifolds}

\subsection{Definitions}
Let $\mathcal F$ denote a ``feature'' space, which correspond to typically high-dimensional measurements made by the imaging system, describing biological function. We are interested in the combined analysis of space and function, and will therefore work with the product space $\mathbb R^d \times \mathcal F$ ($d=2$ or 3).

At macroscopic scale, an image is usually defined, using a continuum approximation, as a function $q: \mathbb R^d \to \mathcal F$. However, in biological imaging, the image values, discretized over pixels, result from the accumulated contributions (counts) of various chemical components collected in the imaged volume and are discrete in nature. Mathematically, ``counting'' is represented using Dirac measures. The elementary Dirac measure $\delta_x \otimes \delta_f$ (for $x\in \mathbb R^d$ and $f\in \mathcal F$), when evaluated at a set $V \times A$ in $\mathbb R^d \times \mathcal F$, returns 1 if $x\in V$ and $f\in A$ and zero otherwise. It can be interpreted as an indication that a ``basic element'' (for example, a protein, or a cell) is observed at location $x$ with feature $f$ (which can be, for example, a protein species or a cell type).
The Diracs can be added, taking 
\[
\mu = \sum_{k=1}^n \delta_{x_k} \otimes \delta_{f_k},
\]
so that $\mu(V\times A)$ counts the number of pairs $(x_k, f_k)$ in the $V\times A$ and can be viewed as a microscopic description  
of the data. 

If $V$ is a small volume, 
the quantity
\[
\rho(V) = \frac1{\vert V\vert } \sum_{k=1}^n \delta_{x_k}(V)
\]
where $\vert V\vert $ is the volume of $V$, measures the density of these elements. The quantity 
\[
\zeta_V = \frac{\sum_{k: x_k\in V}\delta_{f_k}}{\sum_{k=1}^n \delta_{x_k}(V)}
\]
then provides a probability measure on $\mathcal F$ that describes the feature profile of $V$. The continuum approximation corresponding to macroscopic scales is obtained by fixing $x\in \mathbb R^d$ and letting $V_x$ be an infinitesimal neighborhood of $x$ with volume $dx$ and  approximating $\zeta_{V_x}$ with $\zeta_x$:
\begin{equation}
\label{eq:volume.limit}
\mu(V_x \times A) \simeq \zeta_x(A) \rho(x) dx,\quad A\subset \mathcal{F}.
\end{equation}

We introduce the following notation. If $\mathfrak m$ is a measure on $\mathbb R^d$ and $\zeta$ a transition probability from $\mathbb R^d$ to $\mathcal F$ (i.e., a function $x \mapsto \zeta_x$ where $\zeta_x$ is a probability measure on $\mathcal{F}$), we define $\mathfrak m \nototimes \zeta$ as the measure such that
\begin{equation}
    \label{eq:im.var.disintegrated}
    (\mathfrak m\nototimes \zeta) (U \times A) = \int_U \int_A d\zeta_x(f) d\mathfrak m(x)
\end{equation}
for all measurable $U\subset \mathbb R^d$ and $A\subset \mathcal F$. 
The  measure $\mu$ in \cref{eq:volume.limit} is equal to $(\rho \leb) \nototimes \zeta$, where $\leb$ denotes Lebesgue's measure, with $d\leb(x)=dx$. 
We will refer to measures on $\mathbb R^d \times \mathcal F$ as {\em image varifolds}, and this concept provides a unified representation of microscopic and macroscopic scales.
Defining  ``space-feature functions'' as mappings $F: \mathbb R^d\times \F \to \mathbb R$, image varifolds are linear operators on the set of such functions, with notation
\[
(\mu\mid F) = \int_{\mathbb R^d\times \F} F(x,f) d\mu(x,f).
\]

We emphasize
that any function $q: \mathbb R^d \to \mathcal F$ can be considered as an image varifold $\mu_q$
such that 
$q \mapsto \mu_q $ provides a one-to-one representation of measurable functions and for any 
$F$, one has
\begin{equation}
(\mu_q\mid F) = \int_{\mathbb R^d} F(x, q(x)) dx
\, . \label{linear-operator-image-varifold}
\end{equation}

We use diffeomorphisms to transform image varifolds
and to define geodesics in the space of image varifolds.
Diffeomorphisms act on functions $q:\mathbb R^d \to \mathcal F$ as $\varphi \star q = q\circ \varphi^{-1}$ with \eqref{linear-operator-image-varifold} implying
\begin{align*}
(\mu_{\varphi\star q}\mid F) &= \int_{\mathbb R^d} F(x, q\circ \varphi^{-1}(x)) dx\\
&= \int_{\mathbb R^d} \vert D\varphi \vert \, F(\varphi(x), q(x)) dx \\
&= (\mu_q\,\vert \,\vert D\varphi\vert \, F(\varphi(\cdot), \cdot))
\end{align*}
where $\vert D\phi\vert $ is the absolute value of the Jacobian determinant of $\phi$.
This suggests defining the action of a diffeomorphism $\varphi$ on an image varifold as an extension of $\mu_q \to \mu_{\varphi\star q}$, simply letting, for a general varifold $\mu$:
\begin{equation}
    \label{eq:action}
(\varphi\star \mu \, \vert \, F(\cdot, \cdot)) = (\mu\,\vert \,\vert D\varphi\vert \, F(\varphi(\cdot), \cdot)).
\end{equation}

The following definition summarizes this discussion.
\begin{definition}
\label{def:image.varifold}
Let $\mathcal F$ be equipped with a $\sigma$-algebra making it a measurable space. A $d$-dimensional image varifold is a measure on the set $\mathbb R^d \times \mathcal F$.

If $\mathfrak m$ is a measure on $\mathbb R^d$ and $\zeta$ a transition probability from $\mathbb R^d$ to $\mathcal F$, the measure  $\mathfrak m\nototimes \zeta$ in \cref{eq:im.var.disintegrated} is called an image varifold in disintegrated form.

Diffeomorphisms of $\mathbb R^d$ act on image varifolds through the action defined in \cref{eq:action}.   
\end{definition}

\begin{remark}
In the decomposition of \cref{eq:im.var.disintegrated}, $\zeta_x$ only needs to be defined for $x$ in the support of $\mathfrak m$. If $\zeta$ is a fixed measure on $\mathcal{F}$, $\mathfrak m\nototimes \zeta$ is the product measure between $\mathfrak m$ and $\zeta$, in which case we will prefer the standard notation $\mathfrak m \otimes \zeta$. 

\end{remark}

The previous discussion provides examples of image varifolds in disintegrated form.
First the ``continuum image varifolds''
takes the form:
\begin{subequations}
\begin{align}
&
\mu = (\rho \, \leb) \nototimes \zeta \
\intertext{ and the image varifold $\mu_q$ for function $q:\mathbb R^d \to \mathcal F$ has $\mathfrak m$ Lebesgue's measure and $\zeta_x = \delta_{q(x)}$.
The discrete image varifold
}
&\mu = \sum_{k=1}^n \delta_{x_k} \otimes \delta_{f_k}, \ x_k \in \mathbb R^d, f_k \in \mathcal F 
\label{eq:im.var.exp.2}
\end{align}
\end{subequations}
has $\mathfrak m = \sum_{k=1}^n \delta_{x_k}$ and $\zeta_x = \delta_{q(x)}$ where $q: \mathbb R^d \to \mathcal F$ is any function such that  $q(x_k) = f_k$ for $k=1, \ldots, n$. Indeed, we can write, for any functions $F: \mathbb R^d \to \mathbb R$ and $G:\mathcal{F} \to \mathbb R$,
\begin{align*}
\big(\mathfrak m \nototimes \zeta\mid FG\big) &=  \int_{\mathbb R^d} \int_{\mathcal F} F(x) G(f) d\zeta_x(f) d\mathfrak m(x) \\
& =  \int_{\mathbb R^d} F(x) G(q(x)) d\mathfrak m(x) \\
& = \sum_{k=1}^n  F(x_k) G(f_k).
\end{align*}

\medskip

Image varifolds are the main focus of this paper, with a primary goal to develop a numerical approach allowing for their comparison. They have been introduced in \citep{RN21} as a tool for the analysis of spatially resolved transcriptomic images, in combination with a hierarchical modeling. (We will however only consider a single scale in the present paper.) Varifolds \citep{almgren1966plateau} were introduced as a mathematical representation of surfaces (or more generally of Riemannian manifolds), as measures in on the product space $\mathbb R^3 \times S^2$, where $S^2$ (replaced by a Grassmannian for general manifolds) is the unit sphere in $\mathbb R^3$, in order to facilitate the analysis of variational problems over surfaces. In that original model, the equivalent of $\mathfrak m$ in \cref{def:image.varifold} is the singular measure supported by a surface  $M\subset \mathbb R^3$ and $\zeta_x$ is, for $x\in M$, the Dirac measure at  the normal to $M$ at $x$. Surface varifolds have been introduced in \citep{charon2013varifold} for shape analysis, and used in conjunction with the LDDMM algorithm to develop surface matching methods. 

\begin{remark}
We point out that an alternate action of diffeomorphisms on varifolds can be defined in which the Jacobian determinant is dropped from the right-hand side of \cref{eq:action}. The resulting action (denoted $\varphi_\sharp \mu$) is the push-forward of the measure $\mu$ by $\varphi$. The resulting action on images (here interpreted as densities) is $\varphi_\sharp q = \vert D\varphi^{-1}\vert \, q\circ \varphi^{-1}$. This latter action is the one used in shape analysis to compare curves or surfaces \citep{charon2013varifold}. In our setting, where we need to compare tissues with similar compositions but different sizes, this push-forward action is not appropriate, since, say, expansion results in $\vert D\varphi^{-1}\vert  < 1$ and a reduction of the original density (i.e., a sparsification of cells in tissue), which is undesirable. The  action we choose throughout for image varifolds leaves the magnitude of $q$ unchanged, essentially creating more volume without changing the composition of the tissue using a ``copy and paste'' operation. 
\end{remark}

\subsection{A semi-discrete representation of varifolds}
\Cref{eq:im.var.exp.2} describes a varifold in full discrete form, which is well adapted for numerical computations. In the following, however, it will be convenient to have more flexibility on the image transition probabilities,  allowing them to be non discrete. We still discretize  the spatial domain using Dirac measures, but, in preparation for our mesh model in the next section, we attach these  measures to small subsets of $\mathbb R^d$ and provide them with weights that depend on the volume of these subsets. This results in ``semi-discrete varifolds,''  used throughout,  defined by
\begin{enumerate}[label = (\roman*)]
    \item A finite family, $\Gamma$, of subsets of $\mathbb R^d$
    with a list of ``centers,'' $m_\gamma \in \gamma$, $\gamma\in \Gamma$, with volumes $\vert \gamma\vert $;
    \item A list of weights, $\alpha_\gamma \geq 0$,  $\gamma\in \Gamma$;
    \item A list of probability measures on $\mathcal F$, $\zeta_\gamma$, $\gamma\in \Gamma$;
\end{enumerate}
Our space of image-varifold 
$(\Gamma, m, \alpha, \zeta)$ 
with
action via diffeomorphisms becomes
\begin{subequations}
\begin{align}
 &   
 \label{eq:disc.var}
\mu = \sum_{\gamma\in \Gamma} \alpha_\gamma\,\vert \gamma\vert \, \delta_{m_\gamma} \otimes \zeta_\gamma.
\\
&
\label{eq:action.discrete}    
\varphi\star \mu = \sum_{\gamma\in \Gamma} \alpha_\gamma\, \vert D\varphi(m_\gamma)\vert \,\vert \gamma\vert \, \delta_{\varphi(m_\gamma)} \otimes \zeta_\gamma.
\end{align}
\end{subequations}
We call these varifolds ``semi-discrete'' since we use Dirac measures for the spatial component but not necessarily for the image.
For a space-feature functions they act linearly on functions on $\mathbb R^d\times \mathcal F$:
\[
(\mu \, \vert \, F) = \sum_{\gamma\in \Gamma} \alpha_\gamma\, \vert \gamma\vert \, \int_{\mathcal F} F(m_\gamma, f) d\zeta_\gamma(f).
\]

\subsection{Mesh-based varifolds used for computation}
\label{sec:meshes}
We now specialize further to the situation in which the sets in $\Gamma$ are associated with meshes in $\mathbb R^d$, using simplicial meshes (i.e., tetrahedra in 3D and triangles in 2D). Letting $d$ denote the dimension, (two or three), we define a simplicial family as a collection $\boldsymbol x = (x_i, i\in I)$ of distinct points in $\mathbb{R}^d$ together with a family of $(d+1)$-tuples,  $C = (c = (c_{0}, \ldots , c_{d}) \in I^{d+1})$ of indexes such that the simplices
\[
\gamma_c(\boldsymbol x) = \left\{ \sum_{i=0}^{d} a_i x_{c_i}, a_i\geq 0 , a_0+\dots+a_{d}=1\right\}
\]
have non-empty interior with positive orientation, i.e., their volume is
\begin{equation}
    \label{eq:volume}
\vert \gamma_c(\boldsymbol x)\vert  := \mathrm{det}(x_{c_1} - x_{c_0}, \ldots, x_{c_{d}} - x_{c_0})/d!,
\end{equation}
requiring that the term in the right-hand side is positive.
The simplex centers $m_c$ are
\[
m_c(\boldsymbol x) = \frac1{d+1} (x_{c_0} + \cdots + x_{c_{d}}).
\]

This family forms a simplicial mesh of some subset $D$ of $\mathbb{R}^d$ if the simplices only intersect at faces, edges or vertexes and their union is equal to $D$, but we will not need to enforce this  constraint in this paper.
We let $S = (I, C)$, the collection of indexes and $(d+1)$-tuples, which represents the  structure of the family. We denote the family itself (with a valid instantiation of vertexes) as $(I, C, \boldsymbol x) = (S, \boldsymbol x)$.

Let $\mathcal F$ denote a feature space, as above. 
\begin{definition}
\label{def:tetrahedral}
A simplicial image varifold structure is given by a simplicial family $(S,\boldsymbol x)$ with $S=(I,C)$, a family non-negative numbers $(\alpha_c, c\in C)$ and a family of probability  measures on $\mathcal F$,  $\boldsymbol \zeta = (\zeta_c, c\in C)$, with everything summarized as $\mathcal T = (S,\boldsymbol x, \boldsymbol{\alpha}, \boldsymbol \zeta)$.
The associated image varifold and the result of its transformation by a diffeomorphism $\phi$ are
\begin{subequations}
\begin{align}
\label{eq:mu.T}
\mu_{\mathcal T} &= \sum_{c\in C} \alpha_c\, \vert \gamma _c({\boldsymbol x})\vert \, \delta_{m_c({\boldsymbol x})} \otimes \zeta_c
\\
\label{eq:phi.mu.T}
\varphi\star \mu_{\mathcal T} &= \sum_{\gamma\in \Gamma} \alpha_c\, \vert \gamma_c(\boldsymbol x)\vert \,\vert D\varphi(m_c({\boldsymbol x}))\vert \, \delta_{\varphi(m_c({\boldsymbol x}))} \otimes \zeta_c
\end{align}
Define the action of $\phi$ on $\mathcal T$ by $\varphi\cdot \mathcal{T} = (S,\varphi(\boldsymbol x),\boldsymbol{\alpha}, \boldsymbol \zeta)$. Then we have
\begin{align}
\label{eq:action.discrete-mesh}
\mu_{\varphi\cdot \mathcal T} &= \sum_{\gamma\in \Gamma} \alpha_c\, \vert \gamma_c(\varphi(\boldsymbol x))\vert \, \delta_{m_c(\varphi({\boldsymbol x})))} \otimes \zeta_c \\
& \nonumber
\simeq \varphi\star \mu_{\mathcal T} \ ,
\end{align}
\end{subequations}
with the approximations $\varphi(m_c(\boldsymbol x)) \simeq m_c(\varphi(\boldsymbol x))$ and
\[
\vert D\varphi(m_c(\boldsymbol x))\vert  \simeq \frac{\vert \gamma_c(\phi(\boldsymbol x))\vert }{\vert \gamma_c(\boldsymbol x)\vert }.
\]

\end{definition}


\section{LDDMM for discrete image varifolds}
At the core of geodesic brain mapping is our norm-distance that we define on the space of varifold-brains.
For this we define a family of
varifold norms that measure the size of the difference between elements in the space.
\subsection{Image-varifold LDDMM}
Let $K_1$ and $K_2$ be two positive kernels respectively on $\mathbb R^d$ and $\mathcal F$. This means that $K_1$ is defined on $\mathbb R^d\times \mathbb R^d$ with values in $\mathbb R$ such that, for all $n>0$, all $x_1, \ldots, x_n\in \mathbb R^d$ and $\lambda_1, \ldots, \lambda_n\in \mathbb R$, one has
\[
\sum_{k,l=1}^n \lambda_k\lambda_l K_1(x_k, x_l) \geq 0.
\]
The same condition is assumed for $K_2$, replacing $\mathbb R^d$ by $\mathcal F$. A natural choice for $K_1$, the spatial kernel, is to use radial basis functions (such as Gaussian, or Mat\'ern kernels \citep{Aronszajn-1950, schaback2006kernel, cheney2009course,iske2018approximation}).
Image kernels for categorical features
are provided by positive definite matrices with size equal to the number of features, with entries equal to $K_2(f,g)$ for all pairs $f,g \in \mathcal F$.

Define the varifold inner product by the condition, holding for all $x,y\in \mathbb R^d$ and $f,g\in \mathcal F$, 
\begin{equation}
\langle\delta_x\otimes \delta_f, \delta_y\otimes \delta_g\rangle_{W^*} = K_1(x,y)K_2(f,g).
\label{product-kernel}
\end{equation}
By linearity, this defines a unique inner product between measures over $\mathbb R^d \times \mathcal F$. The notation ``$W^*$'' comes from the fact that this inner product can be interpreted as that associated with the dual space of the reproducing kernel Hilbert space on functions defined on $\mathbb R^d \times \mathcal F$ associated with the tensor product of $K_1$ and $K_2$. Similarly, for two measures $\zeta, \zeta'$ on $\mathcal F$, we will write
\[
\langle \zeta, \zeta ' \rangle_{W_2^*} = \int_{\mathcal F \times \mathcal F} K_2(f, g) d\zeta(f) d\zeta'(g).
\]

Let $V$ be a space of vector fields, i.e., of functions $v: \mathbb R^d\to \mathbb R^d$. We assume that $V$ is equipped with an inner product denoted $\langle \cdot ,\cdot \rangle_V$ and associated norm $\|\cdot\|_V$ and forms furthermore a Hilbert space so that it is complete for its norm topology. We also assume that elements in $V$ have at least one continuous derivative, and more precisely that there exists a constant $A$ such that, for any $x\in \mathbb R^d$ and any $v\in V$,
\begin{equation}
    \label{eq:embed.V}
\vert v(x)\vert  + \vert Dv(x)\vert  \leq A \|v\|_V.
\end{equation}
(Here, we let $\vert Dv(x)\vert $ denote any matrix norm applied to the $d\times d$ differential of $v$.)

The LDDMM (discrete) varifold matching problem is, given two varifolds structures $\mathcal T^{(k)} = (S^{(k)}, \boldsymbol{x}^{(k)}, \boldsymbol{\alpha}^{(k)}, \boldsymbol{\zeta}^{(k)})$, $S^{(k)}=(I^{(k)},C^{(k)}), k=0,1$, template and target respectively,  
the variational problem is:
\begin{problem}

\begin{subequations}
\begin{align}
\label{eq:lddmm.unreduced.0}
 \inf_{\substack{v(\cdot) \in L^2([0,1],V) } }
&\int_0^1 \|v(t)\|_V^2 dt + \frac{1}{\sigma^2} \|\mu_{\varphi(1)\cdot \mathcal T^{(0)}} - \mu_{\mathcal T^{(1)}}\|^2_{W^*}
\\
\text{with} \ \ \ \ \ \ \ \ \ \ \
&\partial_t \varphi(t) = v(t)\circ \varphi(t) \ .
\end{align}
\end{subequations}
\end{problem}
This formulation follows the common pattern of other LDDMM algorithms \citep{Beg-Miller-Trouve-Younes-2005,Vaillant-Glaunes-2005,glaunes-trouve-younes-2004,charon2013varifold,younes2019shapes}. Because the action only affects vertexes, this problem can be reduced using an RKHS argument introduced for the registration of landmarks \citep{joshi-miller-landmark-matching-99,glaunes-trouve-younes-2004}, discrete curves \cite{glaunes2004diffeomorphic} and surfaces \citep{Vaillant-Glaunes-2005}  with an optimization over point-set trajectories.

More precisely, \cref{eq:embed.V} implies that $V$ is a reproducing kernel Hilbert space, and because it is a space of vector fields, its kernel $K_V$ is defined on $\mathbb R^d\times \mathbb R^d$ and takes values in the space of $d\times d$ matrices with real coefficients. This kernel is such that, for any fixed $y\in \mathbb R^d$, $a\in \mathbb R^d$, (i) the vector field $K_V(\cdot, y)a : x\mapsto K_V(x, y)a$ belongs to $V$ and (ii)
\[
\langle K(\cdot, y)a, v\rangle_V = a^T v(y)
\]
for all $v\in V$. One can then show that the optimal $v$ 
takes the form
\begin{equation}
\label{eq:v.reduced}
    v(t,\cdot) = \sum_{i\in I^{(0)}} K(\cdot, z_i(t)) a_i(t)
\end{equation}
where $\boldsymbol{a}(t) = (a_i(t), i\in I^{(0)})$ are free $d$-dimensional vectors and the $d$-dimensional points $\boldsymbol{z}(t) = (z_i(t), i\in I^{(0)})$ are defined in the following reduced problem. 
\begin{problem}
\label{prob:vp.1}
\begin{subequations}
\begin{align}
\label{eq:lddmm}
 \inf_{\substack{a_i(t), t \in [0,1],i\in I^{(0)}} }
 &
\sum_{i,j\in I^{(0)}}\int_0^1 a_i(t)^T K(z_i(t), z_{j}(t)) a_{j}(t) dt \\
& \qquad + \frac{1}{\sigma^2} 
\|\mu_{(
S^{(0)}, \boldsymbol{z}(1), \boldsymbol{\alpha}^{(0)}
, \boldsymbol{\zeta}^{(0)})} - \mu_{\mathcal T^{(1)}}\|^2_{W^*}
\intertext{with}
& \label{eq:state}
\partial_t z_i(t) = \sum_{j\in I^{(0)}} K(z_i(t), z_{j}(t)) a_{j}(t) , \quad z_i(0) = x^{(0)}_i, \ i\in I^{(0)} \ .
\end{align}
\end{subequations}
\end{problem}

\subsection{Gradient of the objective function}
The optimization in \cref{prob:vp.1} is with respect to the trajectories $\boldsymbol{a}(t) = (a_i(t), i\in I^{(0)})$. The optimal diffeomorphism $\varphi(t, \cdot)$ in \cref{eq:lddmm.unreduced.0} is then given by $\varphi(t,x) = y(t)$ where $y(\cdot)$ solves the ODE:
\[
\partial_t y = \sum_{i\in I^{(0)}} K(y(t), z_i(t)) a_i(t)
\]
with $y(0) = x$.

The gradient, with respect to $\boldsymbol{a}(\cdot)$, of the objective function in \cref{eq:lddmm}  is obtained with the adjoint method and works as follows. Introduce a co-state $\boldsymbol p(\cdot) = (p_i(\cdot), i\in I^{(0)})$. Define the Hamiltonian, evaluated at configurations $\boldsymbol {\tilde p}, \boldsymbol {\tilde z}, \boldsymbol {\tilde a}$ (that do not depend on time):
\[
H(\boldsymbol {\tilde p}, \boldsymbol {\tilde z}, \boldsymbol {\tilde a}) = \sum_{i,j\in I^{(0)}} \tilde p_i^T K(\tilde z_i, \tilde z_j) \tilde a_j - \sum_{i,j\in I^{(0)}} \tilde a_i^T K(\tilde z_i, \tilde z_{j}) \tilde a_{j}.
\]
Then the gradient is computed in two steps. One first solves the system: 
\begin{subequations}
\begin{equation}
    \label{eq:h.system}
    \begin{cases}
    \partial_t \boldsymbol z(t) = \partial_{\boldsymbol {\tilde p}} H(\boldsymbol p(t),\boldsymbol z(t),\boldsymbol a(t))
    \\
    \partial_t \boldsymbol p(t) = - \partial_{\boldsymbol {\tilde z}} H(\boldsymbol p(t),\boldsymbol z(t),\boldsymbol a(t))
\end{cases}
\end{equation}
with boundary conditions $\boldsymbol{z}(0) = \boldsymbol{x}^{(0)}$ and:
\begin{equation}
    \label{eq:p1}
\boldsymbol p(1) = - \nabla U(\boldsymbol{z}(1))
\end{equation}
where
\begin{equation}
\label{eq:U}    
U(\boldsymbol{x}) = 
\frac{1}{\sigma^2} \|\mu_{(
S^{(0)}, \boldsymbol{x}, \boldsymbol{\alpha}^{(0)}, \boldsymbol{\zeta}^{(0)})} - \mu_{\mathcal T^{(1)}}\|^2_{W^*}.
\end{equation}
\end{subequations}
The gradient of the objective function is then given by: 
\[
t\mapsto - \partial_{\boldsymbol {\tilde a}} H(\boldsymbol p(t),\boldsymbol z(t) ,\boldsymbol{a}(t))^T \boldsymbol{p}(t).
\]
The details of this computation have been provided in multiple places (see references above) and the only computation that is specific to our discussion is the evaluation of \eqref{eq:p1} on which we now focus.

\subsection{Derivative of the data attachment term in 3D and 2D}
We now examine both the 3D and 2D cases using similar arguments for the computation of the derivative of $U$ in \cref{eq:U}. The computation involves the inward weighted inward normal vectors to the faces of the simplices.
In 3D, for the tetrahedron
$
\gamma_c(\boldsymbol x) = \left\{ \sum_{i=0}^3 a_i x_{c_i}, a_i\geq 0 , \sum_{i=0}^3 a_i=1\right\}$, these vectors are
 \begin{equation}
     \label{3d-normals}
\begin{aligned}
n_{c,0} &= - ( x_{c_2} - x_{c_1}) \times (x_{c_3} - x_{c_1})\\
n_{c,1} &= ( x_{c_2} - x_{c_0}) \times (x_{3} - x_{c_0})\\
n_{c,2} &= -( x_{c_1} - x_{c_0}) \times (x_{c_3} - x_{c_0})\\
n_{c,3} &= ( x_{c_1} - x_{c_0}) \times (x_{c_2} - x_{c_0}) \ .
\end{aligned}
 \end{equation}

In 2D, with  tetrahedra replaced by triangles, normals to triangle edges are defined as follow.
Letting $J = \begin{pmatrix}
0&-1\\1&0
\end{pmatrix}
$, the normals attached to each triangle are
\begin{equation}
\label{2d-normals}
\begin{aligned}
n_{c,0} &= J(x_{c_2}-x_{c_1})\\
n_{c,1} &= J(x_{c_0}-x_{c_2})\\
n_{c,2} &= J(x_{c_1}-x_{c_0})\ .
\end{aligned}
\end{equation}
Both \cref{3d-normals} and \cref{2d-normals} obey the general definition in which, for $j\geq 1$, $n_{c,j}$ is the unique vector such that
\begin{subequations}
\begin{equation}
    \label{dd-normals.1}
n_{c,j}^Tu = (-1)^{j-1} \mathrm{det}\bigg(u, x_{c_{1}} - x_{c_0}, \ldots, x_{c_{j-1}} - x_{c_0}, x_{c_{j+1}} - x_{c_0}, \ldots, x_{c_{d}} - x_{c_0}\bigg)
\end{equation}
for all $u\in \mathbb R^d$, and 
\begin{equation}
    \label{dd-normals.2}
n_{c,0}^Tu = - \mathrm{det}\bigg(u, x_{c_2} - x_{c_1}, \ldots, x_{c_d} - x_{c_1}\bigg)
\end{equation}
for all $u\in \mathbb R^d$. They furthermore satisfy
\begin{equation}
        \label{dd=normals.3}
        \sum_{j=0}^d n_{c,j} = 0.
\end{equation}
(This property can be easily checked for $d=2$ or 3.)
\end{subequations}
.

We now calculate the derivative of the data attachment term.
\begin{proposition}
\label{prop:1}
Let $S^{(k)}=(I^{(k)},C^{(k)}), k=0,1$.
then the derivative of $U$ in \cref{eq:U} with respect to $x_j$ is  
\[
\partial_{x_j} U(\boldsymbol{x}) = \frac2{\sigma^2} \partial_{x_j} \langle \mu_{(S^{(0)},\boldsymbol x, \boldsymbol{\alpha}^{(0)}, \boldsymbol \zeta^{(0)})}, \mu_{(S^{(0)}, \boldsymbol{{\tilde x}}, \boldsymbol{\alpha}^{(0)}, \boldsymbol{\zeta}^{(0)})} - \mu_{\mathcal T^{(1)}} \rangle_{W^*}.
\]
evaluated with $\boldsymbol{\tilde x} = \boldsymbol x$, with
\begin{multline}
\label{partial-derivative-inner-product.0}
\partial_{x_j} \langle \mu_{(S,\boldsymbol x,\boldsymbol{\alpha}, \boldsymbol \zeta)}, \mu_{(S',\boldsymbol x', \boldsymbol{\alpha}', \boldsymbol  \zeta')}\rangle_{W^*} = \\
\sum_{c\in C: j\in c} \sum_{c'\in C'} \alpha_{c} \alpha'_{c'}\, \vert \gamma_{c'}(\boldsymbol x')\vert  \langle \zeta_{c}, \zeta'_{c'}\rangle_{W_2^*}  \bigg(\frac1{d+1}  \vert \gamma_{c}(\boldsymbol x)\vert \,\nabla_1 K_1(m_{c}(\boldsymbol x), m_{c'}(\boldsymbol x')) \\
+ \frac1{d!} K_1(m_{c}(\boldsymbol x), m_{c'}(\boldsymbol x')) n_{c}(x_j)\bigg)
\end{multline}
where $n_{c}(x_j)$ is the 3D normal to the face opposed to $x_j$ in $\gamma({c})$ of Eqn.
 \eqref{3d-normals}.


\end{proposition}
See the appendix for a proof.

\section{Image varifolds and spatial transcriptomic data.}

Spatially resolved transcriptomics probe a large number of targeted mRNA molecules with high-resolution location information  \citep{RN15}. After post-processing, this data can take various forms, for example represented as a 2D image with a large number of channels (associated with the measured gene set), as a list of points in space with attached gene count information (reconstructing single-cell RNAseq information combined with location \citep{RN27}), or simply as a long list of single mRNA molecules with their detected location. We here consider a general representation that includes most situations of interest. We let $\mathcal G$ denote the set of targeted genes, whose size $\vert \mathcal G\vert =N$ can vary from several hundreds to several thousands.

We assume that the input data is a large family indexing a spatial unit  $j\in\mathcal J$ associating a location and a list of genes, in the form $y_{j},  (g_{j, k}, k=1, \ldots n_j)$, indicating that genes $g_{j, 1}, \ldots, g_{j, n_j}$ were detected at location $y_j$ (genes in the list may be repeated). 
A natural representation becomes the number of detections of gene $g$ at location $y_j$, denoted $$(y_j, (\nu_j(g), g\in \mathcal G))\ , j \in \mathcal J \ . $$ 
This representation includes raw spatially resolved transcriptomics data without additional processing with $n_j = 1$ for all $j$, as well as cell-centered data where  $y_j$ is the cell center and the $\nu_j(g)$'s are the gene counts  associated with that cell. Preprocessing steps which cluster the raw data can be associated with this representation,
the example we explore being a pre-analysis identifying single cells and cell types \citep{Moffitteaau5324} (see \cref{sec:cell.label}).


To construct our varifolds $\mu_{\mathcal T} = \sum_{c\in C} \alpha_c\, \vert \gamma _c({\boldsymbol x})\vert \, \delta_{m_c({\boldsymbol x})} \otimes \zeta_c $  we
assume a spatial resolution is given as a length parameter $\lambda$ in $\SI{}{\micro\meter}$ defining a regular mesh that is first built within a bounding box containing the data and then pruned by deleting all simplices that contain no point $y_j$, $j \in \mathcal J$. The pruned mesh provides the components $(I, C, \boldsymbol{x})$ supporting the image varifold.
Our simplicial family $\gamma_c(\boldsymbol{x}),c \in C$ is formed from
the vertices,
from which we define the image weights and the probability laws $(\alpha_c,\zeta_c), c \in C$ on the functional features. There are alternative choices for these features, listed below, each of them leading to specific definitions of $\alpha_c, \zeta_c$ and image kernel $K_2$.
\subsection{Gene Features}
The image weights $\alpha_c$ represent densities (counts per unit volume). There are two options for their definitions, namely
the density of detected mRNA molecules, or 
the density of points $y_j$, $j\in \mathcal J$, which makes sense to consider, e.g.,  if indexes $j\in \mathcal J$ enumerate single cells. In both cases, given that $\zeta_c$ is a probability distribution on features which we first take as the genes $\mathcal F = \mathcal G$ with $f=g \in \mathcal G$, we let $\zeta_c(g)$ be the frequency of counts for gene $g \in \mathcal G$ relative to all the counts 
$y_j\in \gamma_c(\boldsymbol{x})$. This gives:
\begin{eqnarray}
&&
\zeta_c(g) = \frac{\sum_{j: y_j\in \gamma_c(\boldsymbol{x})} \nu_j(g)}{\sum_{\tilde g\in \mathcal G}\sum_{j: y_j\in \gamma_c(\boldsymbol{x})} \nu_j(\tilde g)}
\\
&& \nonumber
\text{with} 
\ \
\left\{
\begin{aligned}
&\alpha_c = \frac{1}{\vert \gamma_c(\boldsymbol{x})\vert }\sum_{\tilde g\in \mathcal G}\sum_{j: y_j\in \gamma_c(\boldsymbol{x})} \nu_j(\tilde g)\,,
\\[0.5em]
\text{or } & \alpha_c = \frac{1}{\vert \gamma_c(\boldsymbol{x})\vert }\vert \{j: y_j \in \gamma_c(\boldsymbol{x})\} \vert \,.
\end{aligned}
\right.
\end{eqnarray}

We note that, with the second choice for $\alpha_c$, one  disregards the information provided by the total number of counts in each cell. If one thinks of $j$ as indexing cells in a tissue, the first choice for $\alpha_c$ relates to the number of counts per volume, and the second to the number of cells per volume.

Since $\mathcal G$ is a finite set, the image kernel is a positive definite matrix $(K_2(g,\tilde g), g, \tilde g\in \mathcal G)$. The simplest choice for it is to use the identity, i.e., $K_2(g,\tilde g) = 1$ if $g=\tilde g$ and 0 otherwise. 

Note that, in this section and in the next one, the set $\mathcal G$ may be replaced by a representative subset (gene panel) without any change to the discussion.

\subsection{ RNA Count Features}
\label{sec:RNA.count}
Define the RNA count space to be features $\mathcal F = [0, +\infty)^{\mathcal G}$, that is, the set of all families $f = (f(g), g\in \mathcal G)$ with $f(g) \geq 0$. In this context, the simplest choice is to let $\zeta_c$ be a Dirac measure at the averaged counts 
with mRNA count density:
\begin{eqnarray}
&&\zeta_c = \delta_{\bar\nu_c}, \ \ \bar\nu_c(g) = \frac{\sum_{j: y_j\in \gamma_c(\boldsymbol{x})}\nu_j(g)}{\vert \{j:y_j\in \gamma_c(\boldsymbol{x})\}\vert } \ ,
\\
&& \nonumber
\text{
with } \ \alpha_c = \frac{1}{\vert \gamma_c(\boldsymbol{x})\vert }\vert \{j: y_j \in \gamma_c(\boldsymbol{x})\} \vert
\ .
\end{eqnarray}
There is a wide range of possible choices for the image kernel $K_2$, since we are working with quantitative data. The Gaussian kernel $K_2(\nu, \nu') = \exp(-\vert \nu-\nu'\vert ^2/2\sigma^2)$ is a standard example. For our experiments in the next section, we use the product of a Euclidean and a Cauchy kernels, namely
\begin{equation}
    \label{eq:cauchy}
K_2(\nu, \nu') = \frac{\nu^T\nu'}{\sigma^2 + \vert \nu-\nu'\vert ^2}.
\end{equation}
Note that,  since elements of $\mathcal F$ are non-negative, the kernel $K_2$ can be computed in log scale, i.e., applied to $\log(1+\nu), \log(1+\nu')$ instead of $\nu, \nu'$.
\subsection{Cell Label Features}
\label{sec:cell.label}
Now examine the features to be cell types where we assume that the input data has been preprocessed to return cell type labels. We let  $\mathcal F = \{\ell_1, \ldots, \ell_p\}$, the label set, and assume that the data is a list $(y_j,L_j)$ for locations and labels for $j\in \mathcal J$. The measure $\zeta_c$ can the be defined as
\begin{eqnarray}
&&\zeta_c(\ell_k) = \frac{\vert \{j: y_j\in \gamma_c(\boldsymbol{x}), L_j=\ell_k\}\vert }{\vert \{j: y_j \in \gamma_c(\boldsymbol{x})\}\vert }
\\
&& \nonumber
\text{
with } \ \alpha_c = \frac{1}{\vert \gamma_c(\boldsymbol{x})\vert }\vert \{j: y_j \in \gamma_c(\boldsymbol{x})\} \vert
\ .
\end{eqnarray}
It is natural to use a kernel for which labels are orthogonal, i.e., $K_2(\ell, \tilde \ell) = 1$ if $\ell=\tilde \ell$ and 0 otherwise for $\ell,\tilde \ell\in\mathcal F$.

\section{Examples}
\subsection{A toy example}
As a first example, we consider two shapes, supported by discs in 2D or balls in 3D with 2D image feature that can be interpreted as the concentration (between 0 and 1) of some molecule in a substrate. A small disc/ball is compared to a larger one, with the molecule, concentrated in the center, occupying a larger volume in the small shape than in the large one. The registration must therefore globally expand the shape while locally contracting the region occupied by the molecule. This is illustrated in \cref{fig:toy.2d} and \cref{fig:toy.3d}. The 2D disc has 0.8K vertexes and 1.5K triangles and the 3D ball has 4K vertexes and 20K tetrahedra.

\begin{figure}
    \centering
    \includegraphics[trim=15cm 3.5cm 15cm 3.5cm, clip,width=0.22\textwidth]{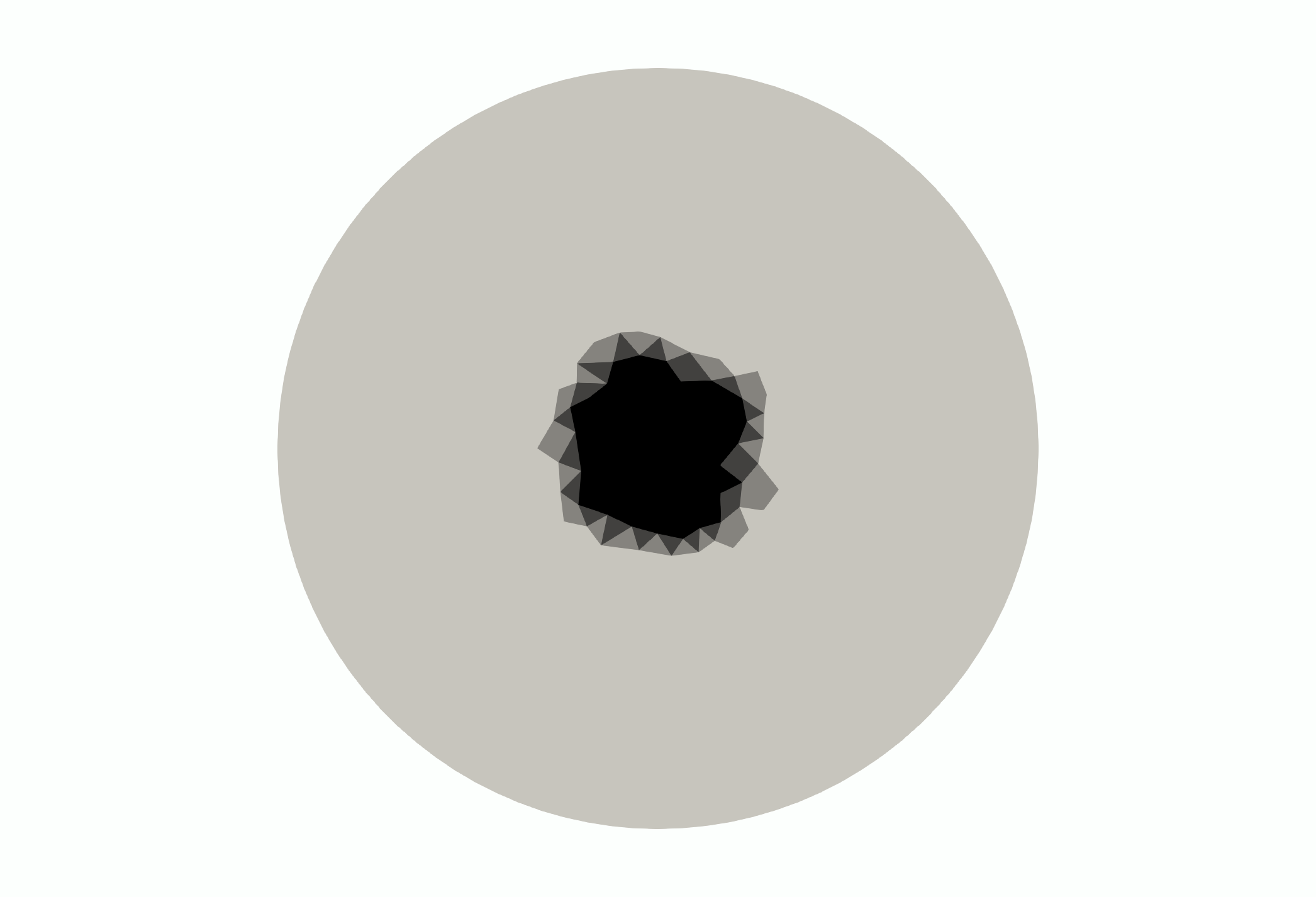}
        \includegraphics[trim=15cm 4cm 15cm 4cm, clip,width=0.22\textwidth]{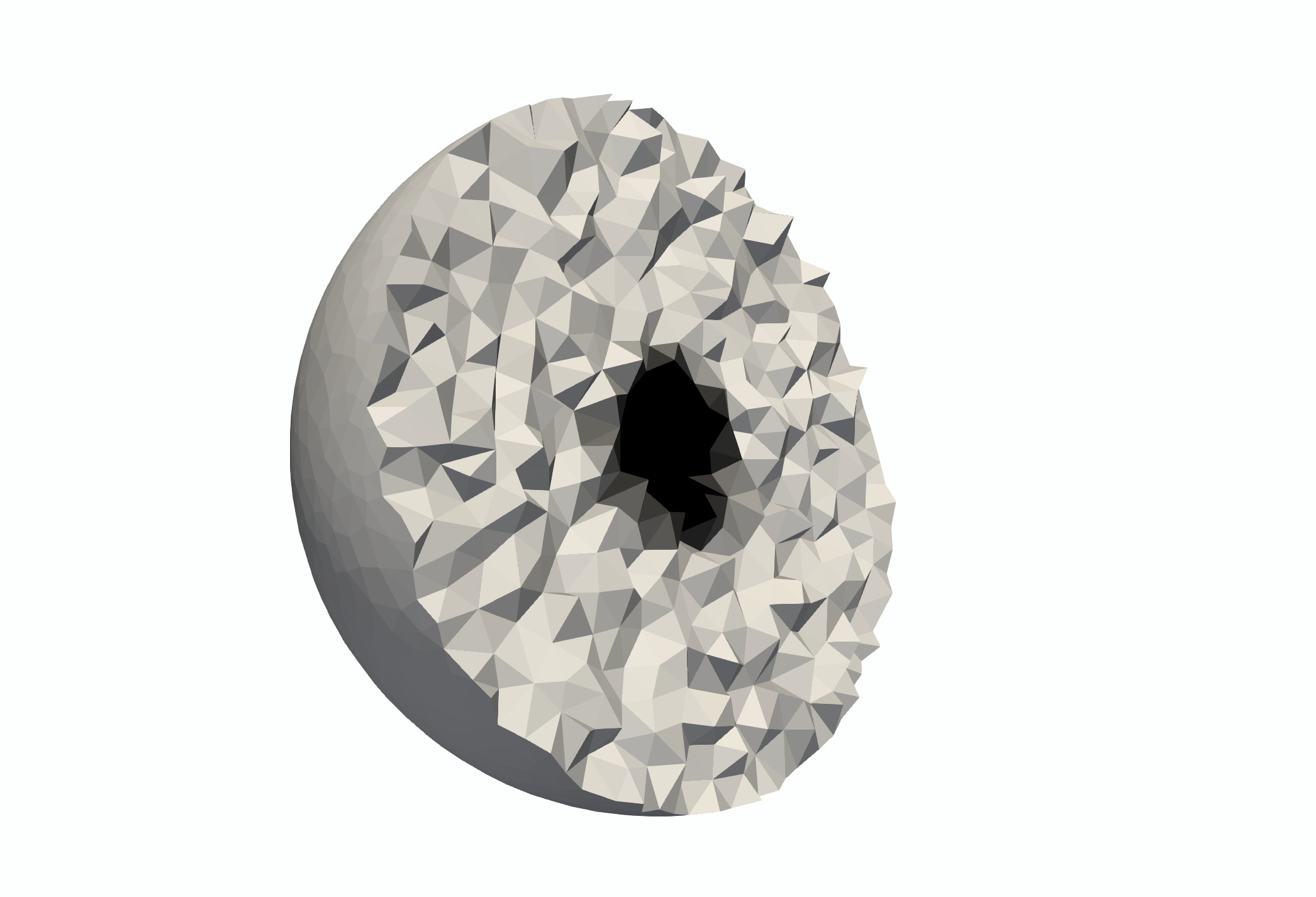}
    \caption{Targets for evolutions in \cref{fig:toy.2d} and \cref{fig:toy.3d}. (The 3D shape is clipped to show interior data.)}
    \label{fig:toy.targets}
\end{figure}

\begin{figure}
    \centering
    \includegraphics[trim=15cm 3.5cm 15cm 3.5cm, clip,width=0.22\textwidth]{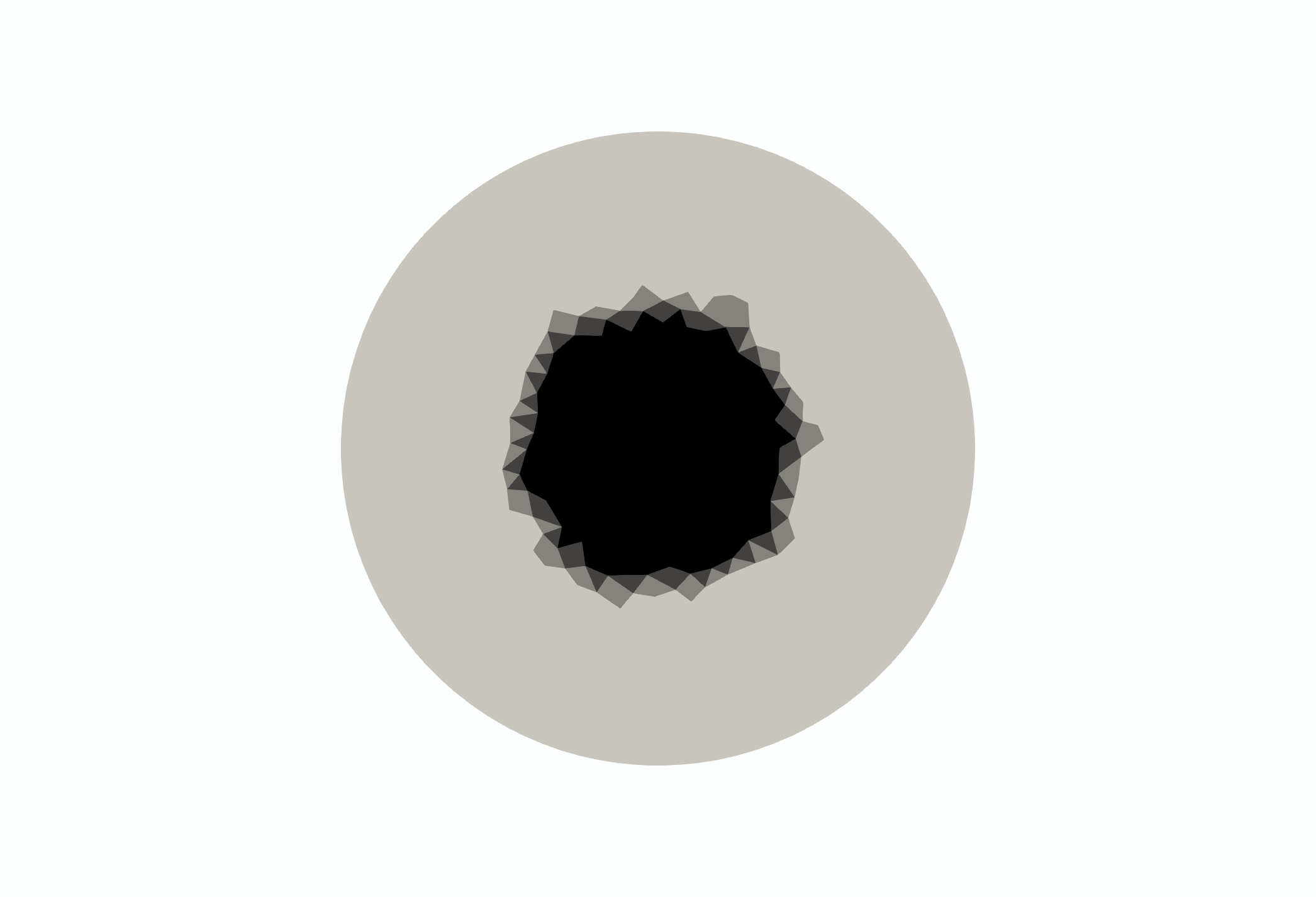}
    \includegraphics[trim=15cm 3.5cm 15cm 3.5cm, clip,width=0.22\textwidth]{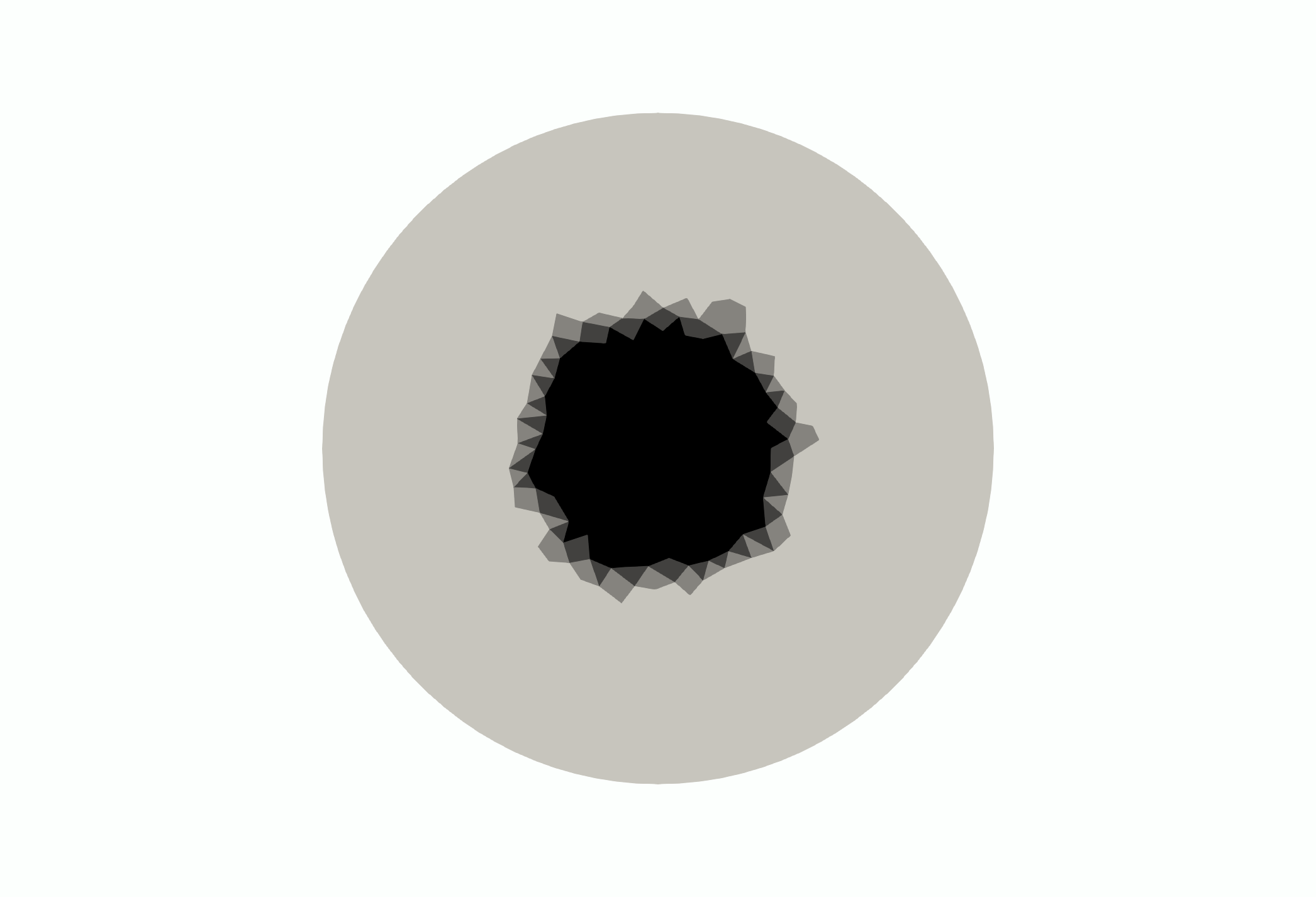}
    \includegraphics[trim=15cm 3.5cm 15cm 3.5cm, clip,width=0.22\textwidth]{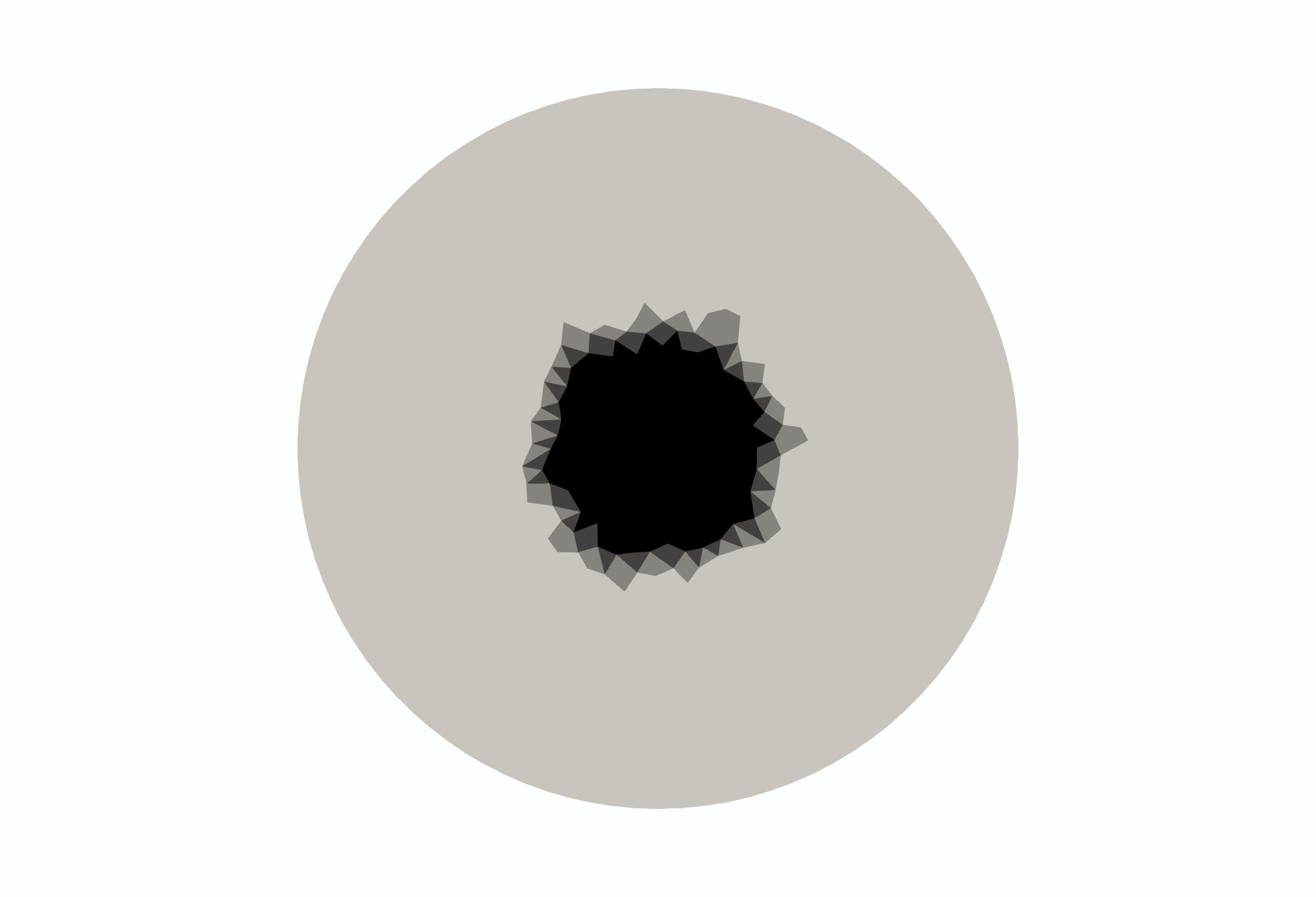}
        \includegraphics[trim=15cm 3.5cm 15cm 3.5cm, clip, width=0.22\textwidth]{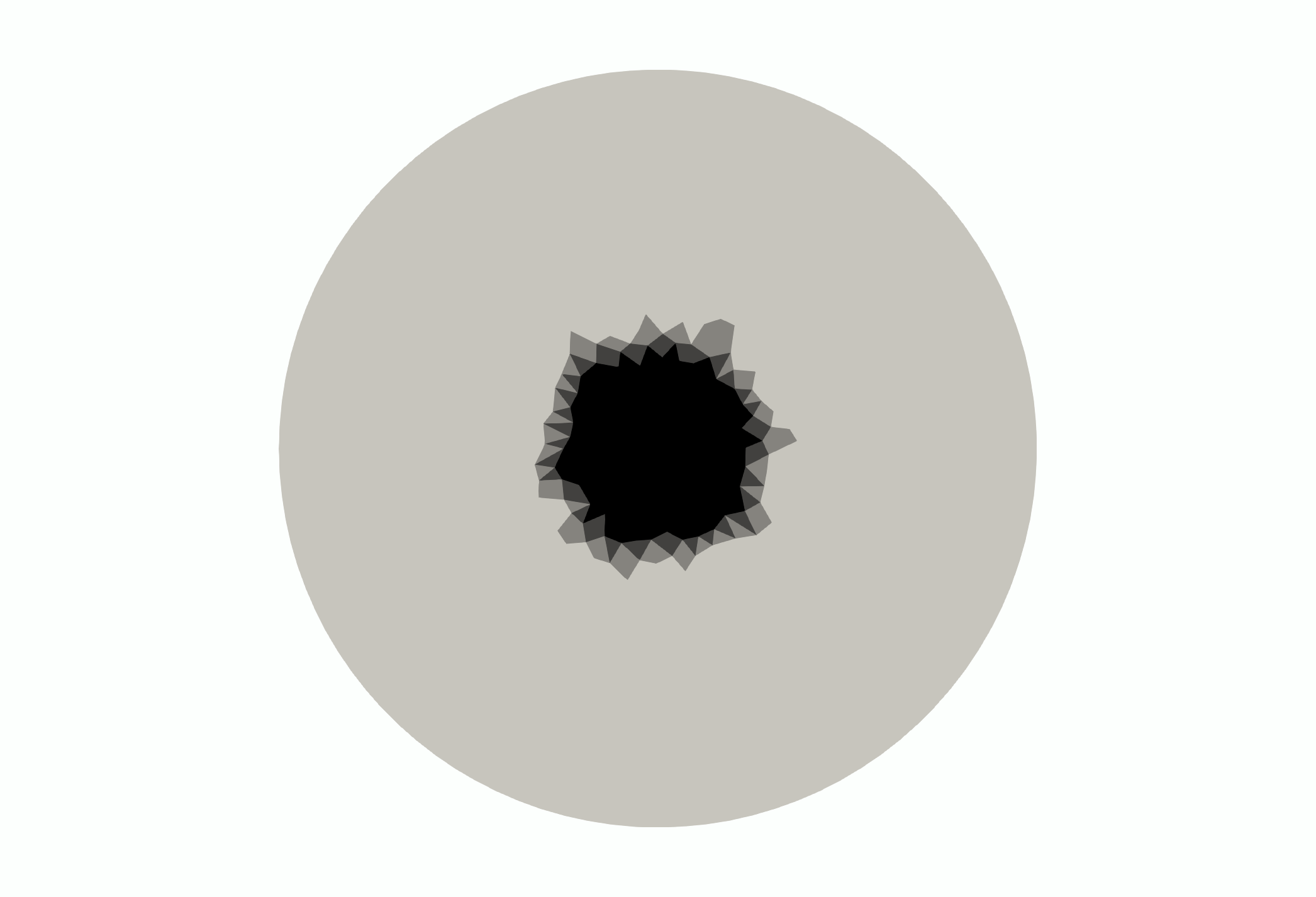}
    \caption{Toy example: 2D. Evolution at times $t=0.0$, $t=0.3$, $t=0.7$ and $t=1.0$. The final image should be compared with the target in \cref{fig:toy.targets}.}
    \label{fig:toy.2d}
\end{figure}

\begin{figure}
    \centering
    \includegraphics[trim=15cm 4cm 15cm 4cm, clip,width=0.22\textwidth]{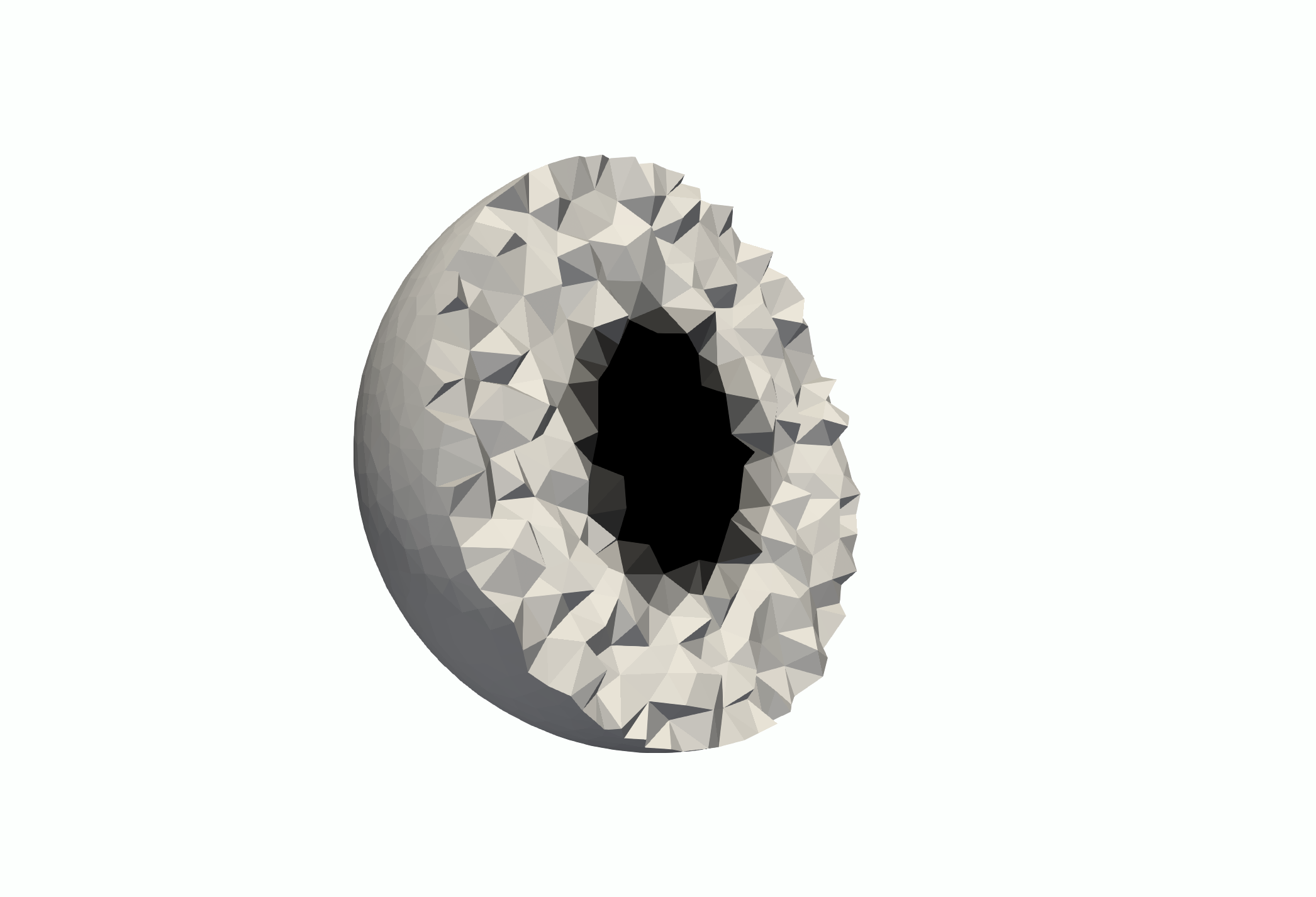}
    \includegraphics[trim=15cm 4cm 15cm 4cm, clip,width=0.22\textwidth]{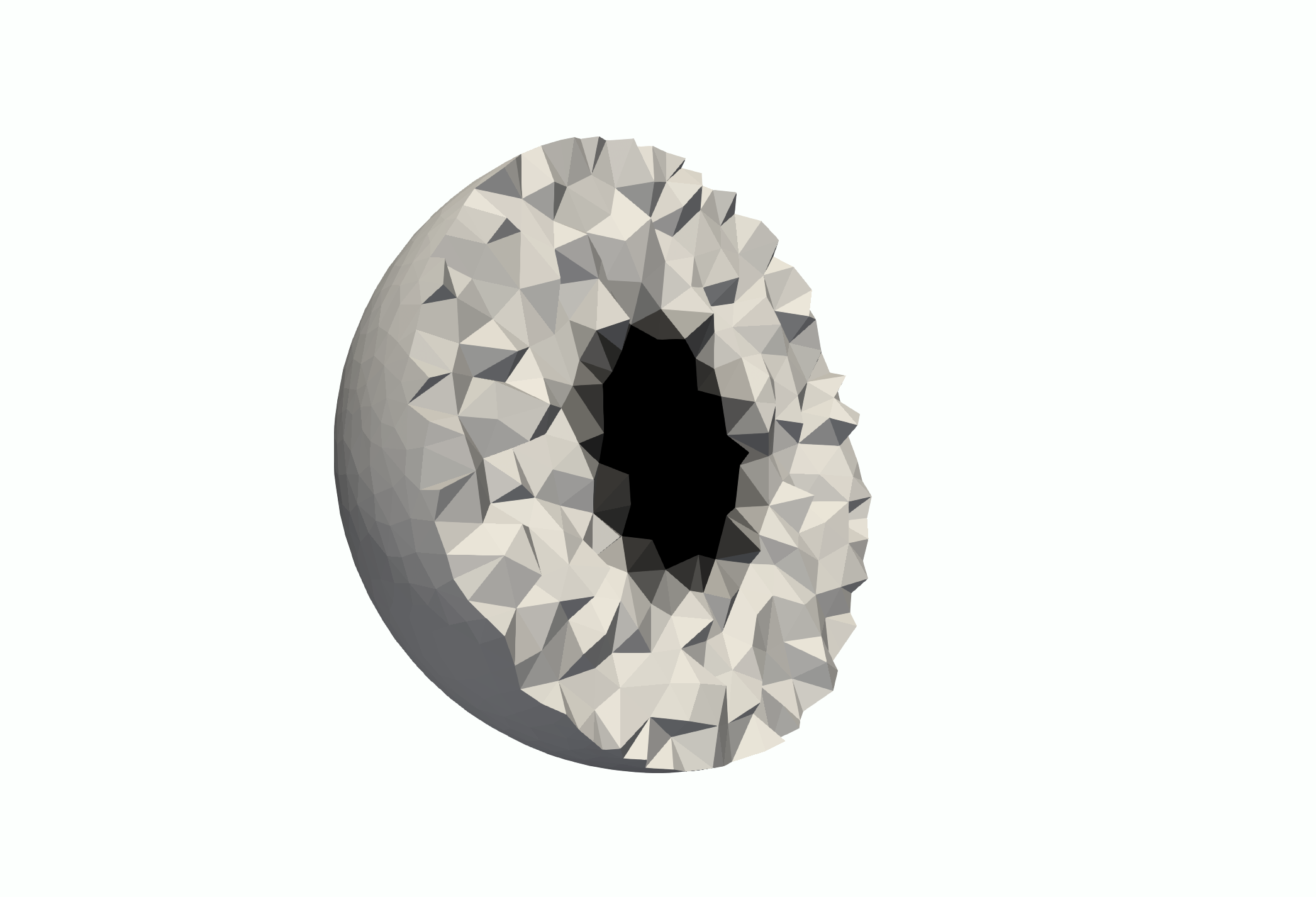}
    \includegraphics[trim=15cm 4cm 15cm 4cm, clip,width=0.22\textwidth]{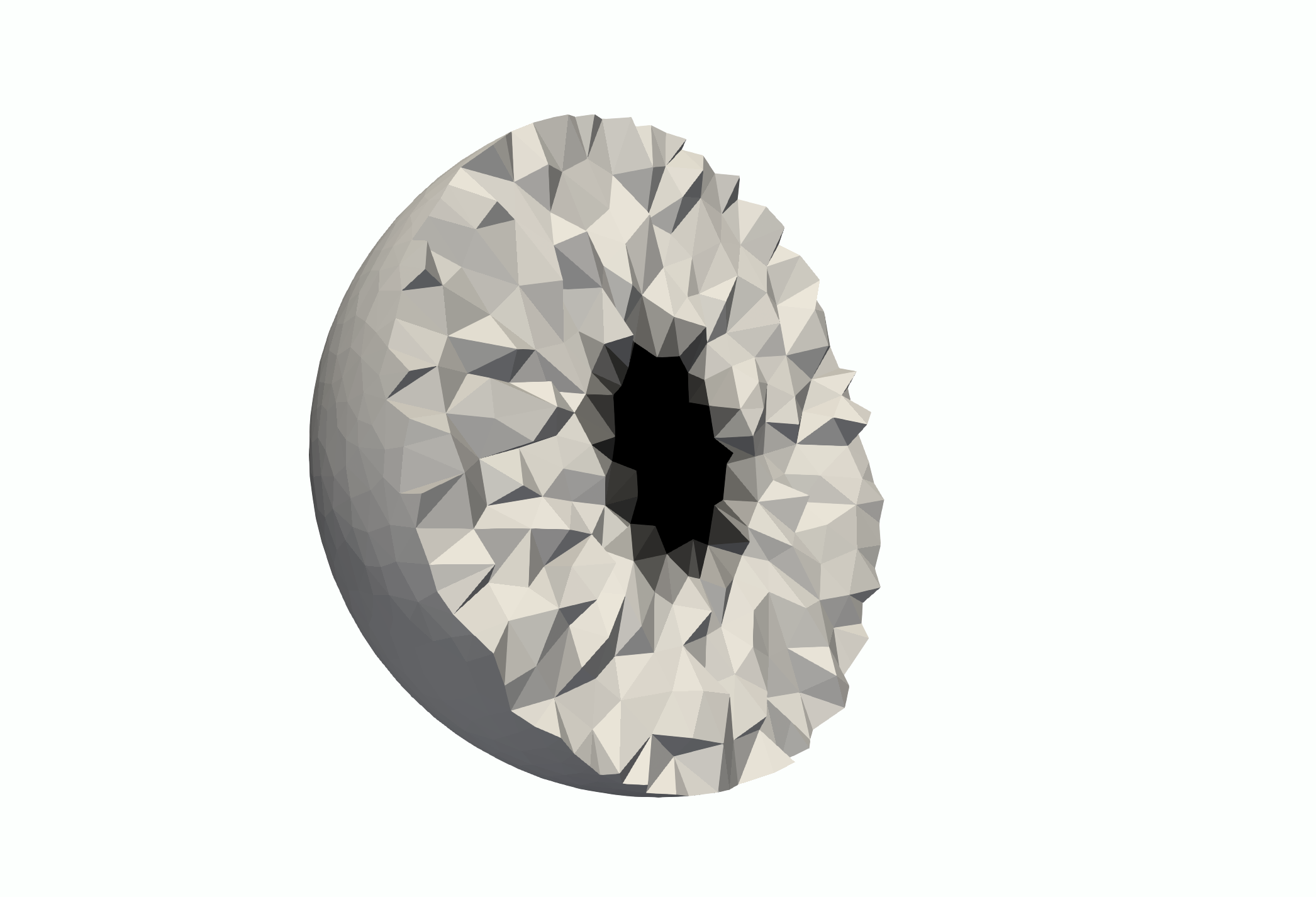}
        \includegraphics[trim=15cm 4cm 15cm 4cm, clip, width=0.22\textwidth]{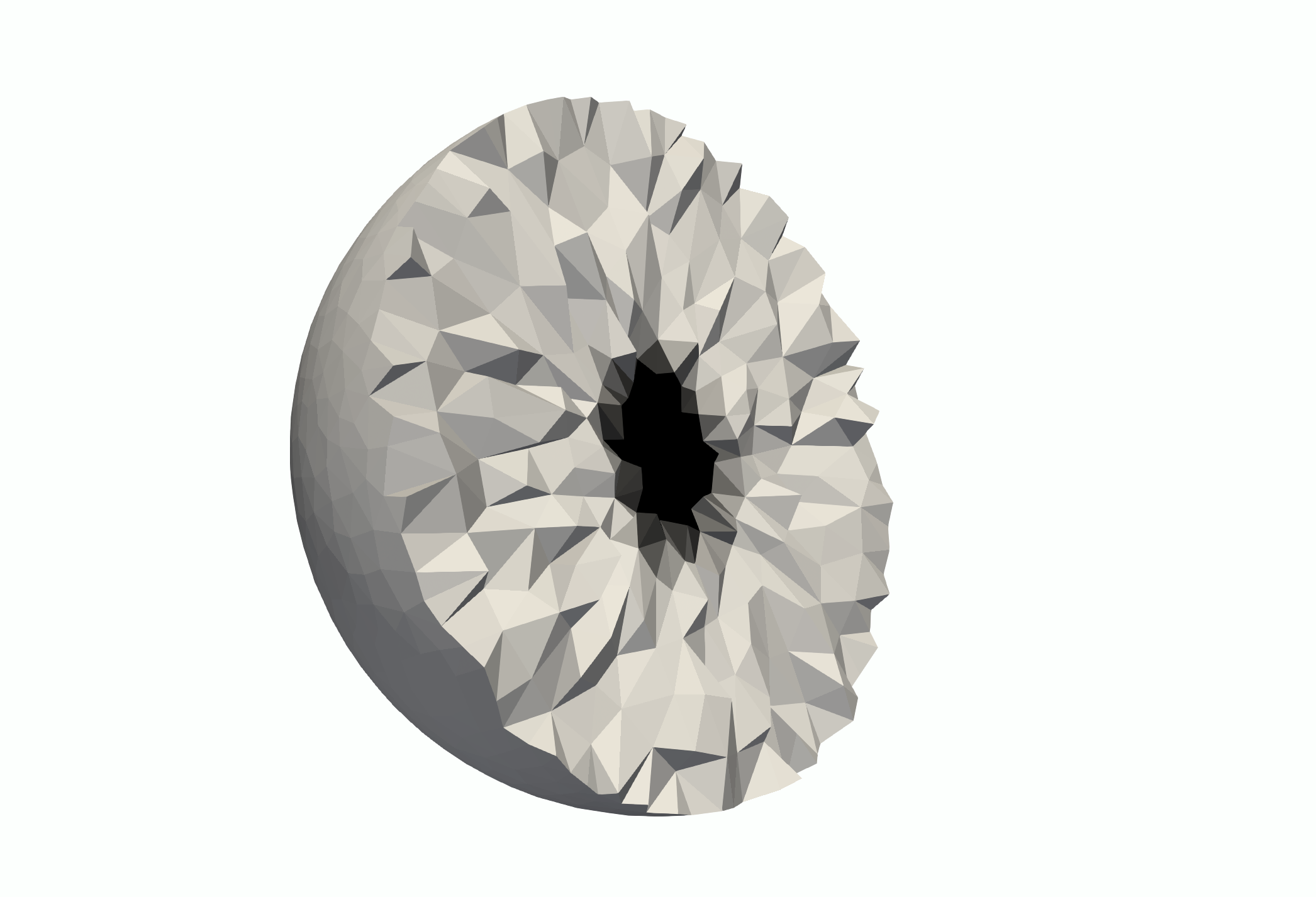}
    \caption{Toy example: 3D. Evolution at times $t=0.0$, $t=0.3$, $t=0.7$ and $t=1.0$. The final image should be compared with the target in \cref{fig:toy.targets}. (Shapes are clipped to show interior data.)}
    \label{fig:toy.3d}
\end{figure}

\subsection{MERFISH image registration}
We illustrate the previous discussion with preliminary based on MERFISH images of mouse brains
\cite{Hongkui-Zeng-2022}. Two-dimensional MERFISH datasets were discretized on grids with spatial resolution $\lambda = \SI{100}{\micro\meter}$. Out of the 700 genes provided for the image, a subset of 10 genes with largest standard deviation was selected to build the image varifolds. We used a Gaussian kernel for $K_1$ and the kernel provided in \cref{eq:cauchy} for $K_2$ after switching to log scale. \Cref{fig:mim-merfish-1stmouse.result} provides images from two brain sections from the same mouse, the first one being used as template and the second as the target for registration.
The top row shows the template, the middle row the target, with the bottom row showing the deformed template sections aligned to the target images.
Figure \ref{fig:mim-merfish-2ndmouse.result}
shows similar results with much great deformations for the same template (top row) but
mapped to a second mouse section (middle row) with the resulting deformed template shown (bottom row).
The deformation grids are shown in \cref{fig:merfish.grid}. We see small deformation when registering the first two sections which come from the same brain and are quite similar, and much stronger changes for the alignment of the first and third sections, which come from different mice and have significant discrepancies.


\begin{figure}[htb]
    \centering
    \includegraphics[trim=4cm 4cm 4cm 4cm, clip, width =0.47\textwidth]{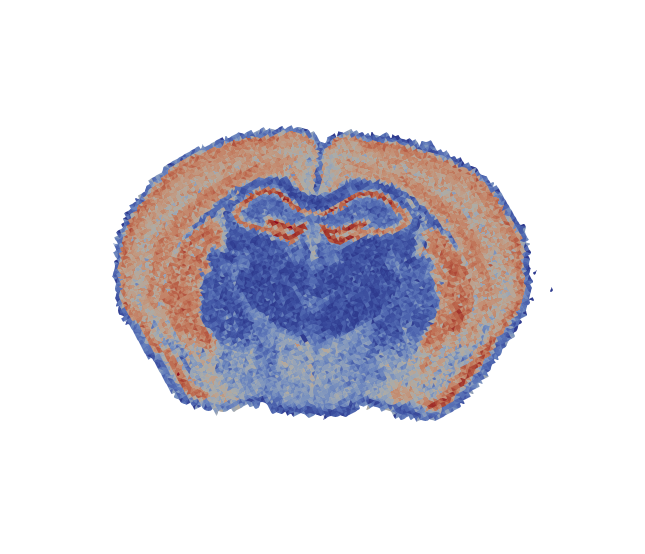}
      \includegraphics[trim=4cm 4cm 4cm 4cm, clip,width =0.47\textwidth]{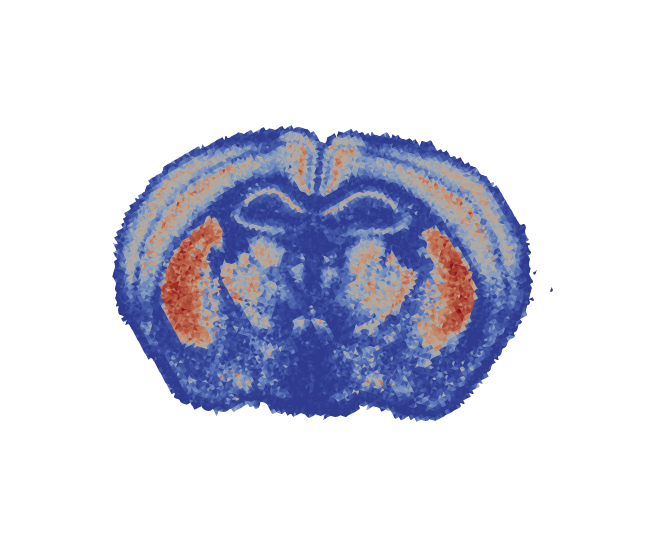}\\
     \includegraphics[trim=4cm 4cm 4cm 4cm, clip,width =0.47\textwidth]{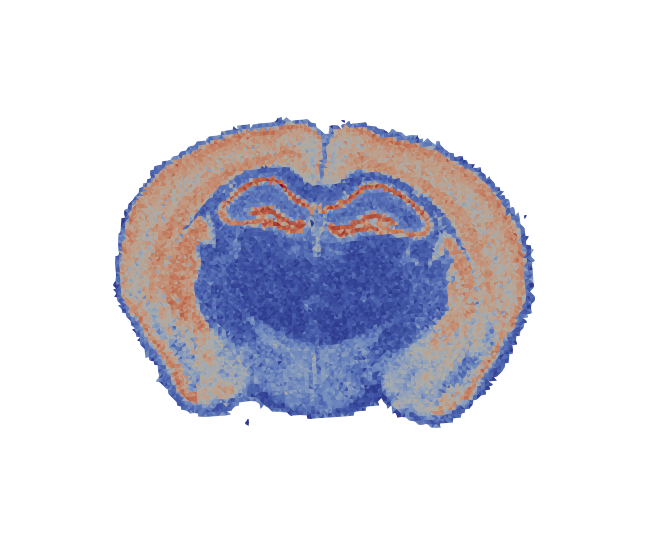}
      \includegraphics[trim=4cm 4cm 4cm 4cm, clip,width =0.47\textwidth]{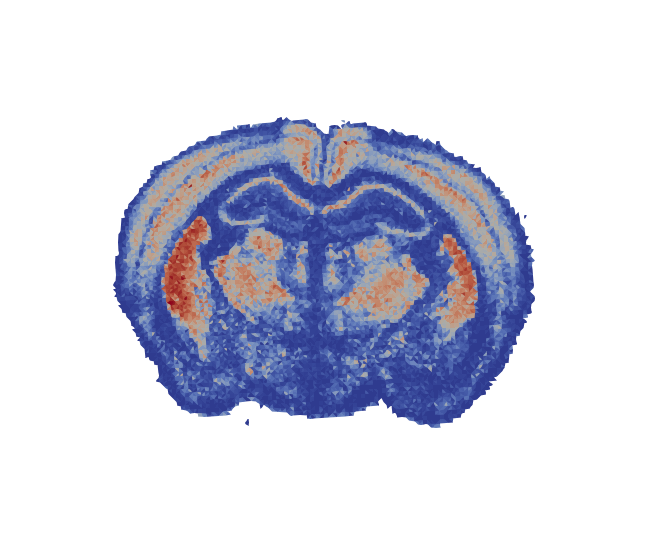}
     \\
      \includegraphics[trim=4cm 4cm 4cm 4cm, clip,width =0.47\textwidth]{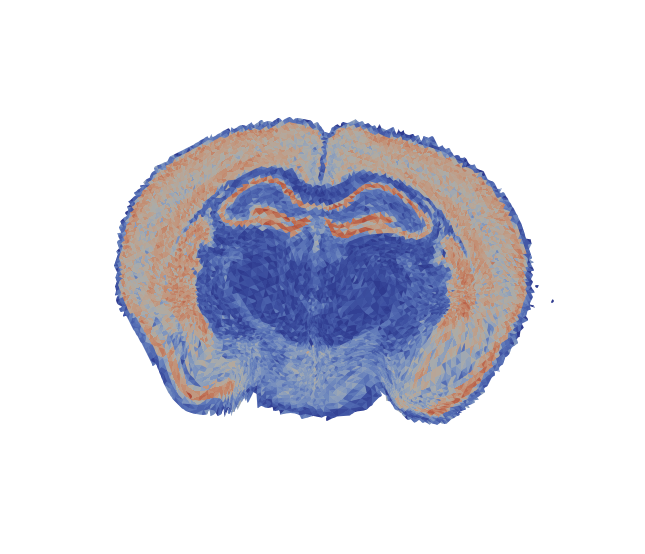}
      \includegraphics[trim=4cm 4cm 4cm 4cm, clip,width =0.47\textwidth]{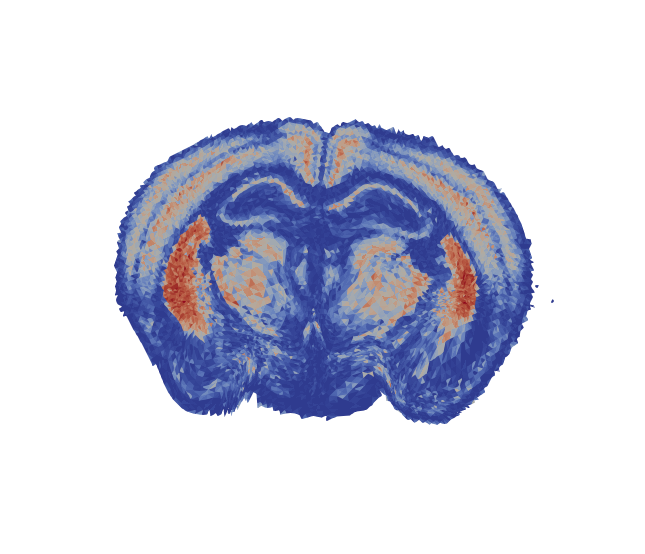}
    \caption{MERFish images: Top row shows template with two genes provided by Atp6ap1l and Satb2 gene counts (in log scale); middle row shows the target. Bottom row shows the template (top) mapped to the target (middle). Data from
\cite{Hongkui-Zeng-2022}.}
    \label{fig:mim-merfish-1stmouse.result}
\end{figure}

\begin{figure}[htb]
    \centering
    \includegraphics[trim=4cm 4cm 4cm 4cm, clip,width =0.427\textwidth]{Figures/merfish_template_im0.png}
  \includegraphics[trim=4cm 4cm 4cm 4cm, clip,width =0.427\textwidth]{Figures/merfish_template_im7.png}\\
  \includegraphics[trim=4cm 4cm 4cm 4cm, clip,width =0.427\textwidth]{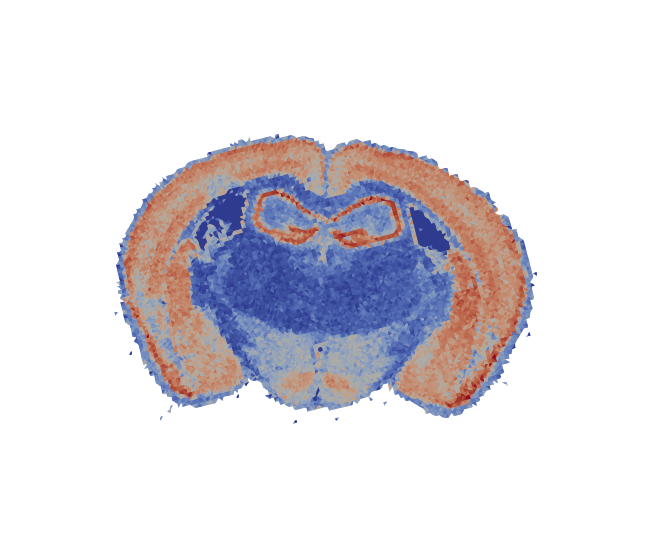}
    \includegraphics[trim=4cm 4cm 4cm 4cm, clip,width =0.427\textwidth]{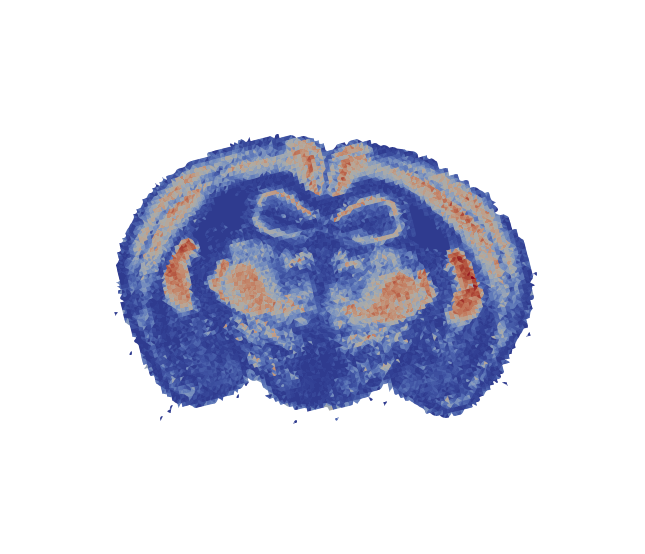}\\
    \includegraphics[trim=4cm 4cm 4cm 4cm, clip,width =0.427\textwidth]{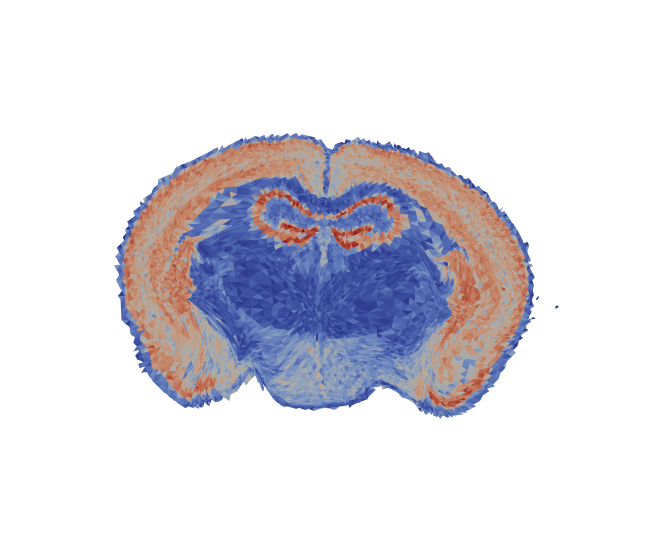}
    \includegraphics[trim=4cm 4cm 4cm 4cm, clip,width =0.427\textwidth]{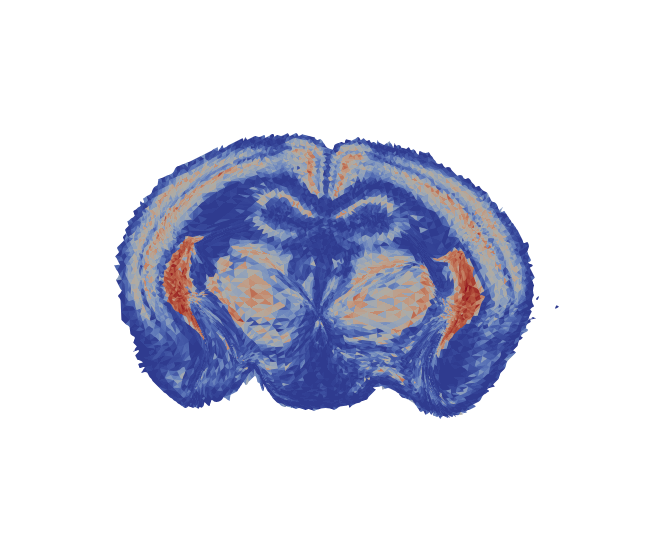}
\caption{
Top row shows template with two genes provided by Atp6ap1l and Satb2 gene counts (identical to top row Figure \ref{fig:mim-merfish-1stmouse.result});
middle row shows target of a second mouse with the two genes Atp6ap1l and Satb2 gene counts (in log scale); bottom row shows the template (top row previous figure \ref{fig:mim-merfish-1stmouse.result}) mapped based on a total of 10 genes of highest variance to the second mouse target.
Data taken from
\cite{Hongkui-Zeng-2022}.
}
    \label{fig:mim-merfish-2ndmouse.result}
\end{figure}


\begin{figure}
    \centering
    \includegraphics[trim=4cm 4cm 4cm 4cm, clip,width =0.47\textwidth]{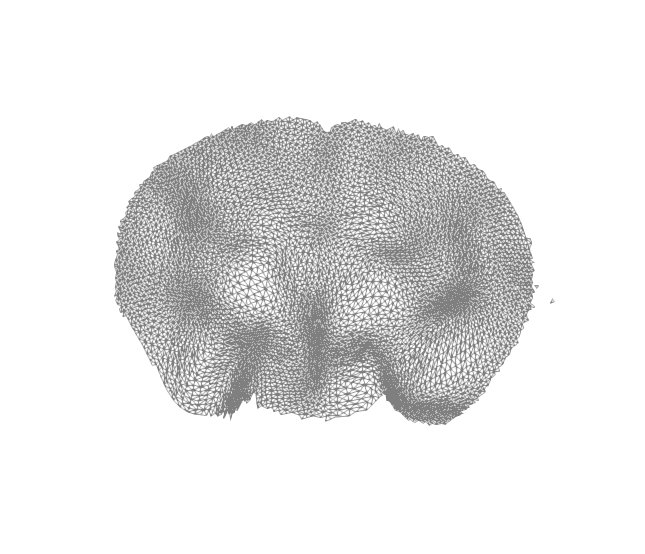}
    \includegraphics[trim=4cm 4cm 4cm 4cm, clip,width =0.47\textwidth]{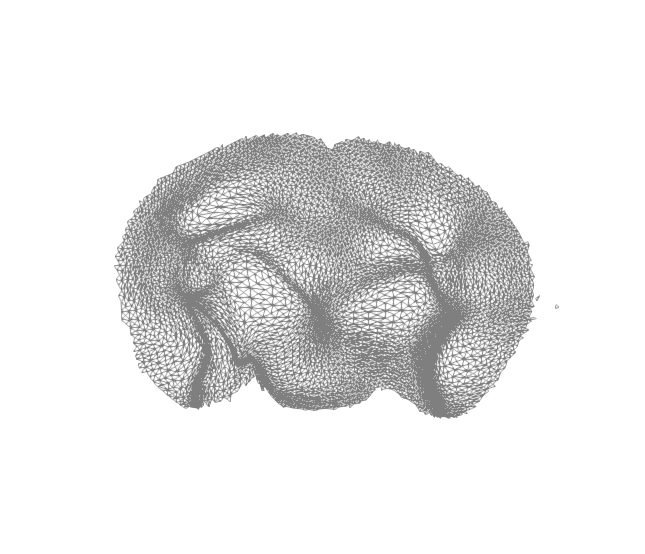}
    \caption{Deformation grids for the two mappings shown above, Figures \cref{fig:mim-merfish-1stmouse.result}, \cref{fig:mim-merfish-2ndmouse.result}. Data from
\cite{Hongkui-Zeng-2022}.
    }
    \label{fig:merfish.grid}
\end{figure}

\section{Atlasing: Crossing modality and Scale}
\subsection{Variational problems}
Transferring genomic, cellular and histological data to atlas coordinates is one of the mainstream examples of crossing modalities and crossing scale.
Atlases (\cref{fig:Allen-atls-mouse}) are by definition often ``cartoons'' \cite{mumford1989optimal} which make sense at the millimeter tissue scales but are used to interpret the finest molecular and particle scales. Similarly our work in human digital pathology brings histological micron scale of markers together with the atlas scales of Mai-Paxinos (\cref{fig:grids}).
We associate to the atlas the features space of cartoon labels $\mathcal L$.
We want to map the high-resolution gene features
with associated  probabilities $ \zeta^{\prime}$ on $\mathcal F$, the set of  micro-scale functional features  representing genomic expression or particle identity, to the tissue scale where we only have the cartoon labels. 

We assume that the varifolds are represented on meshes, as described in \cref{sec:meshes}; the micro-scale fine varifold has $ {\mathcal T}^\prime = ( S^\prime, \boldsymbol{x}^\prime, \boldsymbol{\alpha}^\prime, \boldsymbol{\zeta}^\prime)$,  $ S^\prime = ( I^\prime,  C^\prime)$, with 
\[
\mu_{{\mathcal T}^\prime} = \sum_{c'\in C^\prime}  \alpha_{c'}^\prime \vert \gamma_{c'}(
\boldsymbol{x}^\prime)\vert  \delta_{m_{c'}( \boldsymbol{x}^\prime)} \otimes \zeta_{c'}^\prime.
\]

At the coarse tissue scales, 
we assume that $\mathcal L$ is a small set of region labels forming the atlas features, $\vert {\mathcal L}\vert  \ll \vert {\mathcal F}\vert $. The atlas is  represented as a varifold defined on a mesh $ {\mathcal T} = ( S, \boldsymbol{x}, \boldsymbol{\alpha}, \boldsymbol{\zeta})$,  $ S = ( I,  C)$ with $\zeta_{c}, c \in C$ the atlas feature probability assigning probability to regions. The weights in $\boldsymbol \alpha = (\alpha_c, c\in C)$ are unknown, since they are generally not provided in atlases, but we will assume in the following that they are subject to box constraints in the form $\alpha_c^\mathrm{min} \leq \alpha_c \leq \alpha_c^\mathrm{max}$ for all $c\in C$, where $\alpha_c^\mathrm{min}$ and $\alpha_c^\mathrm{max}$ are known and correspond to prior expectation on the density of molecules or cells. Note that $\alpha_c^{\mathrm{min}} = 0$ and $\alpha_c^{\mathrm{max}} = \infty$ are allowed, and also $\alpha_c^{\mathrm{min}} = \alpha_c^{\mathrm{max}}$, yielding equality constraints, in case the densities are specified or estimated separately.
Through the measures $\zeta_c$, we allow for the specification of a probabilistic atlas, and ``cartoon representations'' (piecewise constant images) are such that simplexes in the same region all have the same probability $\zeta_c$ which is a Dirac.

Each category $\ell\in \mathcal L$ has a specific expression pattern in the tissue, that we represent by a probability distribution on $\F$. Since this  genomic measure feature is not generally available, we propose to estimate it from data, and introduce a family of
parametric measures $(\pi_{\vartheta}, \vartheta \in \Theta)$ on $\mathcal F$ (see examples below). To each label $\ell\in \mathcal L$ is associated a parameter $\theta_\ell\in \Theta$ that need to be estimated. The problem then becomes
to simultaneously estimate the diffeomorphism $\varphi$ of $\mathbb R^d$ mapping the atlas to the micro-scale varifold, the parametrization vector  $\theta  = (\theta_\ell, \ell\in \mathcal L)$, and the region weights.
Importantly, we do not assume that the $\pi_\theta$'s are probabilities measures, and we interpret $\pi_{\theta_\ell}(\mathcal{F})$ as a measure of the density of molecules or cells with label $\ell$.
Using this model, the imputed gene or cellular density at each site $c\in C$ becomes $\alpha_c = \sum_{\ell\in \mathcal L} \zeta_c(\ell)\pi_{\theta_\ell}(\mathcal{F})$  and the probability law at each site $c\in C$ is
$\sum_{\ell\in \mathcal L} \zeta_c(\ell)\pi_{\theta_\ell}/\alpha_c$, yielding a varifold representation of the atlas (with imputed gene or cellular features)
\[
\mu^{\theta}_{{\mathcal T}} = \sum_{c\in C} \vert \gamma_c(\boldsymbol{x})\vert   \delta_{m_{c}(\boldsymbol{x})} \otimes \Big(\sum_{\ell\in \mathcal L} \zeta_c(\ell)\pi_{\theta_\ell}\Big).
\]


The mapping problem is to map
$\mu^{\theta}_{\phi\cdot {\mathcal T}}$ close to $\mu_{{\mathcal T}^\prime}$, with $\varphi$ acting on meshes as defined in \cref{sec:meshes}.
This estimation is performed using alternating minimization, looping over the estimation of $\varphi$ with fixed $\theta$ and the estimation of $\theta$ with fixed $\varphi$ minimizing $\|\mu_{ \varphi(1)\cdot{\mathcal T}}^{\theta} - \mu_{\mathcal T^\prime}\|^2_{W^*}$
with our variational problem for crossing scales.
\begin{problem}
\label{varifold-norm-atlas-problem}
\begin{equation}
\label{eq:lddmm.cross.scales}
\begin{aligned}
 \inf_{\substack{ \theta_\ell, \ell \in \mathcal{L},v(\cdot) \in L^2([0,1],V) } }
&\int_0^1 \|v(t)\|_V^2 dt + \frac{1}{\sigma^2} \|\mu_{\varphi(1)\cdot {\mathcal T}}^{\theta} - \mu_{{\mathcal T}^\prime}\|^2_{W^*}
\\
\text{with} \qquad
&\partial_t \varphi(t) = v(t)\circ \varphi(t) \\
\text{and} \qquad &\alpha_c^{\mathrm{min}} \leq \sum_{\ell\in \mathcal L} \zeta_c(\ell)\pi_{\theta_\ell}(\mathcal{F}) \leq \alpha_c^{\mathrm{max}}, c\in C.
\end{aligned}
\end{equation}
Fixing $v(\cdot) \in L^2[0,1]$, with varifold norm kernel a product form \eqref{product-kernel}, $K=K_1 K_2$,
then we have:
\begin{gather}
\label{eq:varifold-norm-atlas-problem}
\|\mu_{{\varphi(1)\cdot \mathcal T}}^{\theta} - \mu_{{\mathcal T}^\prime}\|^2_{W^*} 
=\\
\begin{aligned}
&
\sum_{c_0,c_1\in C} \sum_{\ell_0, \ell_1\in \mathcal L} 
\vert \gamma_{c_0}\vert \, \gamma_{c_1}\vert  \zeta_{c_0}(\ell_0) \zeta_{c_1}(\ell_1) K_1(m_{c_0}, m_{c_1})
\int_{\mathcal F^2} K_2(f_0, f_1) d\pi_{\theta_{\ell_0}}(f_0)d\pi_{\theta_{\ell_1}}(f_1)\\
&- 2  \sum_{c\in C, c'\in  C^\prime} \sum_{\ell\in \mathcal L} 
\alpha_{c'}^\prime \vert \gamma_{c}\vert \, \vert  \gamma_{c'}^\prime\vert  \zeta_{c}(\ell) K_1(m_{c}, m_{c'}^\prime) \int_{\mathcal F^2} 
K_2(f, f') d\pi_{\theta_{\ell}}(f) d \zeta_{c'}^\prime(f') \ ,
\end{aligned}
\nonumber
\end{gather}
where we have denoted for short $\gamma_c = \gamma_c(\varphi(1)\cdot\boldsymbol{x})$, $m_c = m_c(\varphi(1)\cdot\boldsymbol{x})$, $ \gamma_{c'}^\prime = \gamma_{c'}(\boldsymbol{ x}^\prime)$, $ m_{c'}^\prime = m_{c'}(\boldsymbol{ x}^\prime)$.
\end{problem}

\subsection{Special cases}
\Cref{eq:varifold-norm-atlas-problem} must be minimized in $\theta$ subject to the density constraints in \cref{eq:lddmm.cross.scales}. This generally provides a nonlinear programming problem. However, in some special cases, including those listed below, this problem boils down to quadratic programming (QP).
\begin{enumerate}[label={\bf Example \arabic*.},wide]
    \item 
Take cell types as micro-scale features, so that $( \zeta^\prime_{c'}(f), f\in \mathcal F)$
are the probabilities on each cell type. Let $\Theta$ be the space of positive measures
on $\F$, taking $\pi_\theta (f)=\theta(f)$, $\theta(f) \geq 0$, $f\in\mathcal F$. 

Our problem is to estimate the density of each cell type  in every region $ 
\theta_\ell(f)$, $f\in \F$, $\ell \in \mathcal L$.
Choosing the identity kernel on $\mathcal F$, with $K_2(f,\tilde f) = 1$ if $f=\tilde f$ and 0 otherwise, the
minimization of \cref{eq:varifold-norm-atlas-problem} reduces to:
\begin{multline*}
\inf_{\substack{\theta_\ell, \ell \in \mathcal L\\
\theta_\ell(f) \geq 0, f\in \mathcal F
}}
\sum_{c_0, c_1\in C} \sum_{\ell_0, \ell_1\in \mathcal L}  
\vert \gamma_{c_0}\vert \, \vert \gamma_{c_1}\vert  \zeta_{c_0} (\ell_0) \zeta_{c_1}(\ell_1) K_1(m_{c_0}, m_{c_1})
\sum_{f\in \mathcal F}  {\theta_{\ell_0}}(f){\theta_{\ell_1}}(f)\\
- 2  \sum_{{c}\in C, c'\in C^{\prime}} \sum_{\ell_0\in \mathcal L} 
\alpha_{c'}^\prime \vert \gamma_{c}\vert \, \vert  \gamma_{c'}^\prime\vert  \zeta_{c}(\ell_0) K_1(m_{c}, m_{c'}^\prime)\sum_{f\in\mathcal F} {\theta_{\ell_0}}(f)  \zeta_{c'}^\prime(f),
\end{multline*}
with constraints
\[
\alpha_c^{\mathrm{min}} \leq \sum_{\ell\in \mathcal L}\sum_{f\in \mathcal F} \zeta_c(\ell)\theta_\ell(f) \leq \alpha_c^{\mathrm{max}}, c\in C.
\]

\item
Now consider mRNA counts on genes in $\mathcal G$.
Take a simple model where parameters are $\theta = (\theta(g), g\in \mathcal G) \in [0, +\infty)^{\mathcal G}$ with $\pi_\theta = w_\theta \delta_{\eta_\theta}$,  a weighted Dirac measure, taking $w_\theta = \sum_{g\in\mathcal G}\theta(g)$ and $\eta_\theta(g) = \theta(g)/w_\theta$.
So, to each label $\ell$ is assigned an expression vector $\theta_\ell$ with $w_{\theta_\ell}$ representing the total expression and $\eta_{\theta_\ell}$ a normalized expression. 
Taking
the kernel as Euclidean on $\mathcal F$ (see remark below), $K_2(f,\tilde f) = \sum_{g\in \mathcal G}f(g) \tilde f(g)$, we have 
\[
\int_{\mathcal F} K_2 (f,f') d\pi_{\theta_{\ell_0}}(f) d \zeta_{c'}^\prime (f) = \sum_{g\in \mathcal G} w_{\theta_{\ell_0}}\eta_{\theta_{\ell_0}}(g) \bar\theta_{c'}(g) = \sum_{g\in \mathcal G} \theta_{\ell_0}(g) \bar\theta_{c'}(g)
\]
where $\bar\theta_{c'}$ is the average expression $\int_{\mathcal F} f d \zeta_{c'}^\prime(f)$. This gives
the minimization of \cref{eq:varifold-norm-atlas-problem} reducing to:
\begin{multline*}
\inf_{\substack{\theta_\ell, \ell \in \mathcal L \\
\theta_\ell(g) \geq 0, g \in \mathcal G
}} 
\sum_{c_0,c_1\in C} \sum_{\ell_0, \ell_1\in \mathcal L}
\vert \gamma_{c_0}\vert \, \vert \gamma_{c_1}\vert  \zeta_{c_0}({\ell_0}) \zeta_{c_1}({\ell_1}) K_1(m_{c_0}, m_{c_1})
\sum_{g \in \mathcal G} \theta_{\ell_0}(g) \theta_{\ell_1}(g)
\\
- 2  \sum_{c\in C, c'\in  C^\prime} \sum_{\ell_0\in \mathcal L} 
\alpha_{c'}^\prime \vert \gamma_{c}\vert \, \vert \gamma_{c'}^\prime\vert  \zeta_{c}(\ell_0) K_1(m_{c}, m_{{c'}}^\prime)
\sum_{g \in \mathcal G} \theta_{\ell_0}(g) \bar \theta_{c'} (g)
\end{multline*}
with constraints
\[
\alpha_c^{\mathrm{min}} \leq \sum_{\ell\in \mathcal L} \sum_{g\in \mathcal G} \zeta_c(\ell)\theta_\ell(g) \leq \alpha_c^{\mathrm{max}}, c\in C.
\]

\end{enumerate}

Even though they were obtained from different models and contexts, the two examples above simplify to almost identical QP problems, respectively in $(\theta_\ell, \ell\in \mathcal L)$. These problems are rephrased explicitly as QP problems below.

\begin{my_algorithm}
For Example 1., let $K_2(f,\tilde f) = 1$ if $f=\tilde f$ and 0 otherwise, defining
\begin{equation}
\label{A-constant}
\begin{cases}
A_{\ell_0, \ell_1} = \sum_{c_0,c_1\in C}  \alpha_{c_0}\alpha_{c_1} \vert \gamma_{c_0}\vert \, \vert \gamma_{c_1}\vert  \zeta_{c_0}(\ell_0) \zeta_{c_1}(\ell_1) K_1(m_{c_0}, m_{c_1})
\\
b_{\ell_0}(f) = \sum_{c\in C, c'\in  C^\prime} \alpha_{c}\alpha_{c'}^\prime \vert \gamma_{c}\vert \, \vert  \gamma_{c'}^\prime\vert  \zeta_{c}(\ell_0) K_1(m_{c}, m_{{c'}}^\prime)\zeta_{c'}^\prime(f), 
\end{cases}
\end{equation}
then $
({\theta_\ell}, \ell\in \mathcal L)$ minimizes 
\begin{equation}
\begin{aligned}
    \label{eq:h.min.1}
    &
\Phi_1(\theta) = \sum_{\ell_0, \ell_1\in \mathcal F} \sum_{f\in \mathcal F} A_{\ell_0, \ell_1} {\theta_{\ell_0}}(f) {\theta_{\ell_1}}(f) - 2\sum_{\ell_0\in \mathcal L}\sum_{f\in\mathcal F} b_{\ell_0}(f) {\theta_{\ell_0}}(f)
\\
& \text{subject to the constraints} \ \
\alpha_c^{\mathrm{min}} \leq \sum_{\ell\in \mathcal L}\sum_{f\in \mathcal F} \zeta_c(\ell)\theta_\ell(f) \leq \alpha_c^{\mathrm{max}}, c\in C.
\end{aligned}
\end{equation}
 \end{my_algorithm}

\begin{my_algorithm}
For Example 2., let $K_2(f,\tilde f) = \sum_{g\in \mathcal G}f(g) \tilde f(g)$, 
$\bar\theta_{c'} = \int_{\mathcal F} f d \zeta_{c'}^\prime$,
taking $A$ as in
\cref{A-constant},  and define
\begin{equation}
b_{\ell_0}(g) = \sum_{c\in C, c'\in C^\prime} \alpha_{c}\alpha_{c'}^\prime \vert \gamma_{c}\vert \, \vert  \gamma_{c'}^\prime\vert  \zeta_{c}(\ell_0) K_1(m_{c}, m_{c'}^\prime)
\bar \theta_{c'} ,
\end{equation}
then $
({\theta_\ell}, \ell\in \mathcal L)$ minimizes 
\begin{equation} 
\label{eq:h.min.2}
\begin{aligned}
   &
    \Phi_2(\theta) = 
\sum_{\ell_0, \ell_1\in \mathcal L} \sum_{g\in \mathcal G} A_{\ell_0, \ell_1} \theta_{\ell_0}(g) \theta_{\ell_1}(g) - 2\sum_{\ell_0\in \mathcal L}\sum_{g\in\mathcal G} b_{\ell_0}( g) \theta_{\ell_0}(g)
\\
&
\text{subject to the constraints} \ \ 
\alpha_c^{\mathrm{min}} \leq \sum_{\ell\in \mathcal L}\sum_{g\in \mathcal G} \zeta_c(\ell)\theta_\ell(g) \leq \alpha_c^{\mathrm{max}}, c\in C.
\end{aligned}
\end{equation}
\end{my_algorithm}

\begin{remark}
The Euclidean kernel used for $K_2$ in Example 2 is degenerate, in the sense that it induces a finite-dimensional RKHS, which can be identified to $\mathbb R^{\mathcal G}$ with the standard Euclidean norm. In this representation, probability measures on $\mathcal F$ are identified with their expectations, so that the metric does not differentiate between a Dirac measure $\delta_f$ and a probability measure $\zeta$ with expectation $f$. This explains why, in this example, the measures $\zeta'_{c'}$ were replaced with their associated average expression.

If $K_1$ is a spatial kernel, the product RKHS associated with $K_1K_2$ is identical to the RKHS  of multivariate functions $x \mapsto F(x)\in \mathbb R^{\mathcal G}$, formed with the $|\mathcal G|$-fold tensor product of scalar RKHS's associated with $K_1$.  Image varifolds associated with a Euclidean image kernel are therefore identified to $|\mathcal G|$-dimensional vector measures. 
\end{remark}

  \begin{figure}
    \centering
    \includegraphics[width=0.8\textwidth]{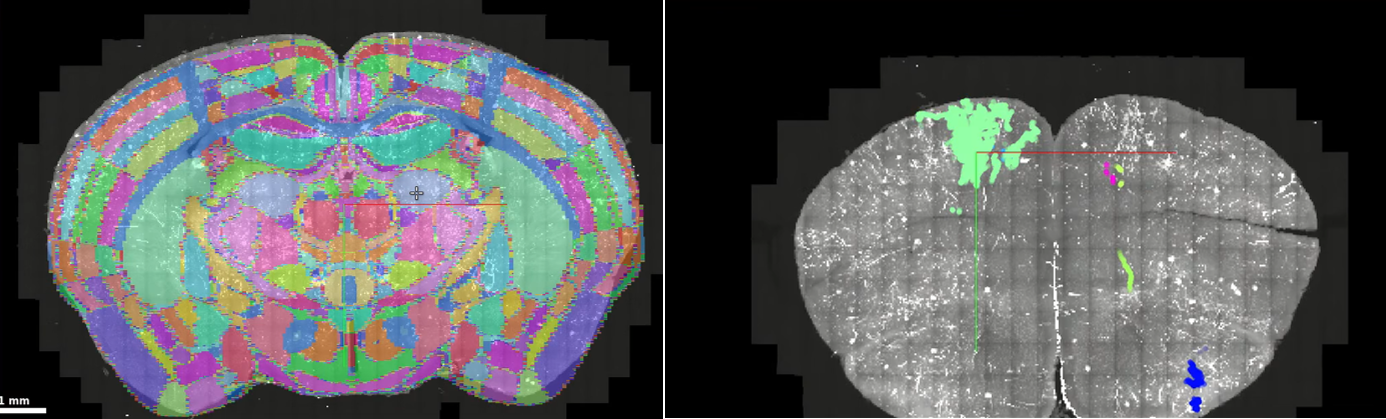}
    \\
    \caption{Depicting Allen Atlas scale (left) and mouse micro scales right. }
    \label{fig:Allen-atls-mouse}
\end{figure}
   \begin{figure}
    \centering
    \includegraphics[width=0.8\textwidth]{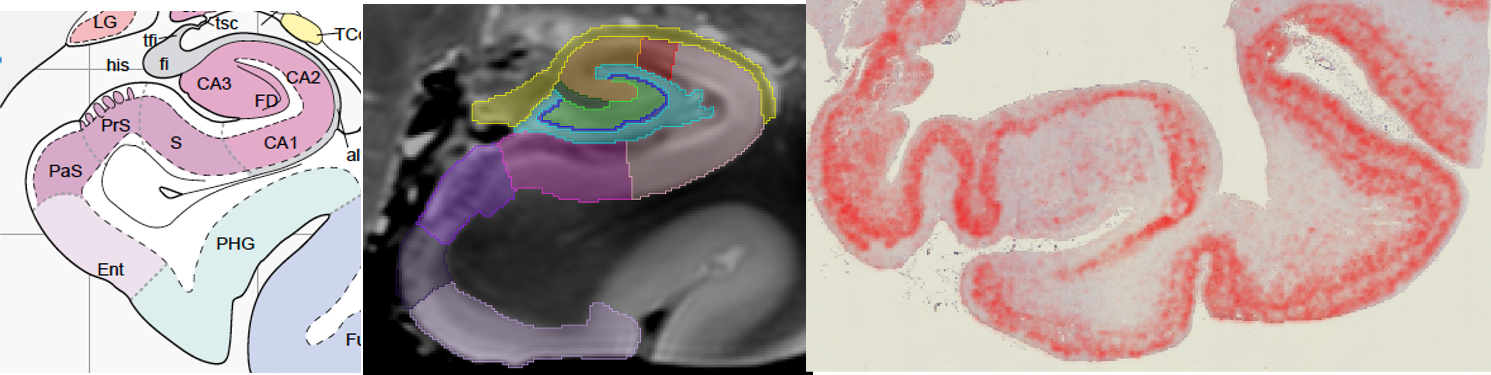}
    \\
    \caption{Mai-Paxinos atlas section and high field MRI with Tau molecular pathology (right). }
    \label{fig:grids}
\end{figure}

\section{Point processes and image varifold}
The realizations of a compound point process $\boldsymbol{N}$ are discrete image varifolds
\[
N = \sum_{k=1}^n \delta_{x_k} \otimes \delta_{f_k}.
\]
Their distributions are specified by a non-negative intensity function $\lambda$ defined on $\mathbb R^d$ and a transition probability $(x, f) \mapsto \zeta(x,f)$  defined on $\mathbb R^d \times \mathcal F$ (so that $\zeta(x,\cdot)$ is for all $x$ a probability on $\mathcal{F}$), such that
\begin{enumerate}[label=(\roman*)]
    \item for $\Omega\subset \mathbb R^d$ and $A \subset \mathcal F$,  $\boldsymbol N(\Omega\times A)$ follows a  a Poisson distribution with parameter 
    \[
    \Lambda^{\boldsymbol{N}}(\Omega\times A) = \int_\Omega \lambda(x) \zeta(x, A) dx.
    \]
    \item $\boldsymbol{N}(\Omega\times A)$ is independent of $\boldsymbol{N}(\Omega'\times A')$ if $(\Omega \times A) \,\cap\, (\Omega'\times A') = \emptyset$.
\end{enumerate}
In the following we will make the abuse of notation $N(\Omega) = N(\Omega\times \mathcal F)$ for the total number of points in $\Omega$ and similarly write $\Lambda^{\boldsymbol{N}}(\Omega) = \Lambda^{\boldsymbol{N}}(\Omega\times \mathcal F)$.
We assume that $\lambda$ is compactly supported, which allows us to represent realizations of $\boldsymbol N$ as finite sums.

If $F$ is a function on $\mathbb R^d \times \mathcal F$, its pairing with $N$ is 
\[
(N\vert F) = \sum_{k=1}^n F(x_k, f_k).
\]
The conditional expectation of $(\boldsymbol N\vert F)$ given the point set $n, x_1, \ldots, x_n$ is
\[
E((\boldsymbol{N} \vert  F)\vert n, x_1, \ldots, x_n) = \sum_{k=1}^n \int_{\mathcal F} F(x_k, f) d\zeta(x_k, f)
\]
which is the pairing of $F$ with the varifold $\sum_{k=1}^n \delta_{x_k} \otimes \zeta(x_k, \cdot)$. One also has
\[
E((\boldsymbol{N} \vert  F)) = \int_{\mathbb R^d} \int_{\mathcal F} F(x, f) d\zeta(x, f) \lambda(x)dx
\]
therefore associated with the varifold $\mu_{\boldsymbol N} = \lambda \otimes \zeta$. We will write $\boldsymbol{N} \sim \cpp(\lambda, \zeta)$ to indicate that $\boldsymbol{N}$ is a compound Poisson point process with intensity $\lambda$ and transition probability $\zeta$.

If $\boldsymbol N = \cpp(\lambda, \zeta)$ and $\varphi$ is a diffeomorphism, we define $\varphi\star \boldsymbol N := \cpp(\lambda\circ \varphi^{-1}, \zeta(\varphi^{-1}(\cdot), \cdot))$, so that 
\begin{align*}
(\mu_{\varphi\star \mathbf N}\vert F) &= \int_{\mathbb R^d} \int_{\mathcal F} F(x, f) d\zeta(\varphi^{-1}(x), f) \lambda(\varphi^{-1}(x))dx
\\
&= \int_{\mathbb R^d} \int_{\mathcal F} \vert d\varphi(x)\vert  F(\varphi(x), f) d\zeta(x, f) \lambda(x)dx \\
&= (\varphi\star \mu_{\mathbf N}\vert F).
\end{align*}

We now design a test statistic to assess whether the observations $N_1, N_2$ of two independent compound Poisson processes result from  $\boldsymbol{N}_1 = \varphi_1 \star \boldsymbol{N}$ and $\boldsymbol{N}_2 = \varphi_2\star \boldsymbol{N}$, where $\boldsymbol N = \cpp(\lambda, \zeta)$ and $\varphi_1$ and $\varphi_2$ are diffeomorphisms. ($\boldsymbol{N}$ is  therefore a ``template'' compound Poisson process.)   We here assume that $\varphi_1$ and $\varphi_2$ are known, and ignore the bias resulting from the fact that they have possibly been estimated using a registration procedure also involving $N_1$ and $N_2$. 

Fix subsets $\Omega \subset\mathbb R^d$ and $A\subset \mathcal F$, and assume that $\lambda$ and $\zeta$ take constant values, $\bar \lambda$ and $\bar \zeta$ on $\Omega$ (so that $\bar \lambda$ is a non-negative number and $\bar \zeta$ is a  measure on $\mathcal F$).
 One then has 
\[
\Lambda^{\boldsymbol{N}}(\Omega\times A) = \int_{\Omega} \lambda(x) \zeta(x, A) dx = \bar \lambda\, \bar \zeta(A)\, \vert \Omega\vert 
\]
and, under the assumptions above: 
\[
\Lambda^{\boldsymbol{N}_i}(\Omega_i\times A) = \Lambda^{\boldsymbol{N}}(\Omega\times A) \frac{\vert \Omega_i\vert }{\vert \Omega\vert } = \Lambda^{\boldsymbol{N}}(\Omega\times A)  \chi_{i},\quad i=1, 2,
\]
with $\Omega_i=\varphi_i(\Omega)$ and $\chi_{i} = \vert \Omega_i\vert \big/\vert \Omega\vert $. So, the null hypothesis is, for given $\Omega, A$: 
\[
H_0(\Omega, A): \frac{\Lambda^{\boldsymbol{N}_1}(\Omega_1\times A)}{\chi_{1}} = \frac{\Lambda^{\boldsymbol{N}_2}(\Omega_2\times A)}{\chi_{2}}, 
\]
and the  alternative hypothesis is
\[
H_1(\Omega, A): \frac{\Lambda^{\boldsymbol{N}_1}(\Omega_1\times A)}{\chi_{1}} \neq \frac{\Lambda^{\boldsymbol{N}_2}(\Omega_2\times A)}{\chi_{2}}.
\]
We will use the likelihood ratio test statistic to compare the two hypotheses.
Letting $\hat p = N_1(\Omega_1\times A)/(N_1(\Omega_1\times A)+N_2(\Omega_2\times A))$ and $p = \chi_{1}/(\chi_{1}+\chi_{2})$, it is given by
\begin{align}
&T(\Omega, A)%
= (N_1(\Omega_1\times A)+N_2(\Omega_2\times A)) \ D(\pi_{\hat p} \vert \vert  \pi_{p}) \ ,
\\ \ \ \ \ \ \ \ \ \ \ \ \ \ \ \  
& \ \ \ \ \ \ \ \ \ \ \ \ \ \  \nonumber \text{with} \ \ D(\pi_{\hat p} \vert \vert  \pi_{p}) =\left( \hat p \log \frac{\hat p}{p} + (1-\hat p) \log \frac{1-\hat p}{1-p}
\right) 
\end{align}

This test statistic can be computed over partitions $(\Omega_c, c\in C)$ of the support of $\lambda$ (small enough to justify the constancy assumption). If one takes $A = \mathcal F$, the collection  $T(\Omega_c, \mathcal F)$ focus on point density only. Using in addition a partition $(A_1, \ldots, A_m)$ of $\mathcal F$ (obtained, for example, as clusters interpreted as cell types), we obtain a complete family of statistics $T(\Omega_c, A_j)$ that provides a high-dimensional analysis of the differences between the observed realizations.


\section{Discussion}
The family of algorithms presented here provides the basis for future mapping technologies that allow for the representation of massive lists of molecular and cellular descriptions of the human body with the tissue scales of radiological and pathological imaging. 
Unifying the molecular and image scales represents an important step forward in brain mapping.
The central representation is the brain as a varifold measure defined on the direct product of space and function.

The algorithms described allow for the molecular computational anatomy mapping program to continue in the vein of D'Arcy Thompson, computing normed distances between brains.
A basic principle calculates similarity by acting diffeomorphisms which transforms one brain onto the other measuring the size of the transformation.
Central to the theory proposed here is the action which we describe as ``copy and paste,'' preserving the density of the quantized objects as space is transformed. This emphasizes the representation as containing two objects, the density $\rho$ on $\mathbb R^d$ and the field of conditional distributions $(\zeta_x, x \in \mathbb R^d)$ representing function over space.
The varifolds norms introduced for placing the varifold measures of brain space into a normed-space score both the density measure as well as function measure.

The varifold brainspace represents both space and function. 
Because the transformations defined act on space, the variation of the norm with respect to the group action becomes the variation of space through the varifold space kernel weighted by the direct inner product measuring alignment of the function measures.

Interestingly the varifold action we derive makes the molecular scale representation consistent with the tissue scale representation associated to MRI imaging and atlasing at 100 micron - 1 millimeter scale.
We explicitly define several features including RNA and cell-centered features.
In all the cases the features are represented as empirical probability laws over the RNA or cell identity feature spaces.

As part of the atlasing method
we examine several algorithms for transferring the high resolution gene features to the atlas tissue scales by inferring the gene features.
We demonstrate that this carries us into a family of quadratic programming problems in which the imputed feature laws are constrained
to be probability measures.

We also examine the family of optimal test statistics for the spatial transcriptomic setting and show that the Kullback-Lieber divergence plays a central role in characterizing discriminability. The KL-distance is calculated between the empirical feature laws $\zeta_x, x \in \mathbb R^d  $ under different hypotheses for the brain measures.

\appendix
\section{Proof of \cref{prop:1}}
We repeat the statement of the proposition for convenience.
\begin{proposition*}
Let $S^{(k)}=(I^{(k)},C^{(k)}), k=0,1$.
then the derivative of $U$ in \cref{eq:U} with respect to $x_j$ is  
\[
\partial_{x_j} U(\boldsymbol{x}) = \frac2{\sigma^2} \partial_{x_j} \langle \mu_{(S^{(0)},\boldsymbol x, \boldsymbol{\alpha}^{(0)}, \boldsymbol \zeta^{(0)})}, \mu_{(S^{(0)}, \boldsymbol{{\tilde x}}, \boldsymbol{\alpha}^{(0)}, \boldsymbol{\zeta}^{(0)})} - \mu_{\mathcal T^{(1)}} \rangle_{W^*}.
\]
evaluated with $\boldsymbol{\tilde x} = \boldsymbol x$, with
\begin{multline}
\label{partial-derivative-inner-product.1}
\partial_{x_j} \langle \mu_{(S,\boldsymbol x,\boldsymbol{\alpha}, \boldsymbol \zeta)}, \mu_{(S',\boldsymbol x', \boldsymbol{\alpha}', \boldsymbol  \zeta')}\rangle_{W^*} = \\
\sum_{c\in C: j\in c} \sum_{c'\in C'} \alpha_{c} \alpha'_{c'}\, \vert \gamma_{c'}(\boldsymbol x')\vert  \langle \zeta_{c}, \zeta'_{c'}\rangle_{W_2^*}  \bigg(\frac1{d+1}  \vert \gamma_{c}(\boldsymbol x)\vert \,\nabla_1 K_1(m_{c}(\boldsymbol x), m_{c'}(\boldsymbol x')) \\
+ \frac1{d!} K_1(m_{c}(\boldsymbol x), m_{c'}(\boldsymbol x')) n_{c}(x_j)\bigg)
\end{multline}
where $n_{c}(x_j)$ is the 3D normal to the face opposed to $x_j$ in $\gamma({c})$ of Eqn.
 \eqref{3d-normals}.


\end{proposition*}

\begin{proof}
We compute the variation of  $\langle \mu_{(S, \boldsymbol x,\boldsymbol{\alpha}, \boldsymbol \zeta)}, \mu_{(S',\boldsymbol x', \boldsymbol{\alpha}', \boldsymbol \zeta')}\rangle_{W^*}$ with respect to a perturbation $\boldsymbol h = (h_i, i\in I)$ on the vertices $\boldsymbol x$ at $\epsilon=0$ of the variation:%
\begin{align*}
&\partial_\epsilon
\langle \mu_{(S, \boldsymbol x+\epsilon\boldsymbol h,\boldsymbol \alpha,\boldsymbol \zeta)}, \mu_{(S',\boldsymbol x', \boldsymbol\alpha',\boldsymbol \zeta')}\rangle_{W^*}\vert _{\epsilon=0}\\
&\ \ \ \ \ = \partial_\epsilon\Big(
\sum_{c\in C,c'\in C'} \alpha_{c} \alpha'_{c'} \vert \gamma_{c}(\boldsymbol x + \epsilon\boldsymbol h)\vert  \, \vert \gamma_{c'}(\boldsymbol x')\vert \, \langle \zeta_{c}, \zeta'_{c'}\rangle_{W_2^*}K_1(m_{c}(\boldsymbol x+\epsilon\boldsymbol h), m_{c'}(\boldsymbol x'))\Big)
_{\big\vert _{\epsilon=0}} \ .
\end{align*}
Differentiating gives two terms
given by the derivative of the kernel and the derivative of the volume term (defined in \cref{eq:volume}).
The derivative of the kernel is
\[
\partial_\epsilon K_1(m_{c}(\boldsymbol x+\epsilon\boldsymbol h), m_{c'}(\boldsymbol x'))\vert _{\epsilon=0}=
\frac1{d+1} \sum_{j=0}^d \nabla_1 K_1(m_{c}(\boldsymbol x), m_{c'}(\boldsymbol x'))^T h_{{c_{j}}} . 
\]
The derivative of the determinant in \cref{eq:volume} gives: 
\begin{align*}
& \partial_\epsilon \vert \gamma_{c}(\boldsymbol x + \epsilon\boldsymbol h)\vert 
\\ &= 
\frac1{d!}\sum_{j=1}^d \mathrm{det}\bigg(x_{c_1} - x_{c_0}, \ldots, x_{c_{j-1}} - x_{c_0}, h_{c_j}-h_{c_0}, x_{c_{j+1}} - x_{c_0}, x_{c_d} - x_{c_0}\bigg)\\
&= \frac1{d!}\sum_{j=1}^d (-1)^{j-1} \mathrm{det}\bigg(h_{c_j}-h_{c_0}, x_{c_1} - x_{c_0}, \ldots, x_{c_{j-1}} - x_{c_0},  x_{c_{j+1}} - x_{c_0}, x_{c_d} - x_{c_0}\bigg)\\
&= \frac1{d!}\sum_{j=1}^d (h_{c_j}-h_{c_0})^T n_{c,j}\\
&= \frac1{d!}\sum_{j=0}^d h_{c_j}^T n_{c,j}
\end{align*}
where the last two equations use \cref{dd-normals.1} and \cref{dd=normals.3}, respectively.

Collecting terms involving $h_j$ gives the variation:
\begin{align}
\label{eq:epg}
&&\frac1{d+1} \sum_{c\in C,c'\in C'} \alpha_{c} \alpha'_{c'} \vert \gamma_{c}(\boldsymbol x)\vert \, \vert \gamma_{c'}(\boldsymbol x')\vert  \langle \zeta_{c}, \zeta'_{c'}\rangle_{W_2^*}\,\nabla_1 K_1(m_{c}(\boldsymbol x), m_{c'}(\boldsymbol x'))^T \Big(\sum_{j=0}^{d} h_{c_{j}}\Big) \\
\nonumber
&& \ \ \ \ \ \ \ \ \ + \frac1{d!}
\sum_{{c\in C},{c'}\in C'} \alpha_{c} \alpha'_{c'} \vert \gamma_{c'}(\boldsymbol x')\vert  \langle \zeta_{c}, \zeta'_{c'}\rangle_{W_2^*} K_1(m_{c}(\boldsymbol x), m_{c'}(\boldsymbol x')) \Big(\sum_{j=0}^d n_{c_j}^Th_{c_{j}}\Big)\ .
\end{align}
Removing dependence on the perturbation direction gives the partial derivative 
\begin{multline*}
\partial_{x_j} \langle \mu_{(S,\boldsymbol x,\boldsymbol{\alpha}, \boldsymbol \zeta)}, \mu_{(S',\boldsymbol x',\boldsymbol{\alpha}', \boldsymbol \zeta')}\rangle_{W^*} = \\
\sum_{{c}\in C: j\in {c}} \sum_{{c'}\in C'}\alpha_{c} \alpha'_{c'} \vert \gamma_{c'}(\boldsymbol x')\vert  \langle \zeta_{c}, \zeta'_{c'}\rangle_{W_2^*}  \Big(\frac1{d+1}  \vert \gamma_{c}(\boldsymbol x)\vert \,\nabla_1 K_1(m_{c}(\boldsymbol x), m_{c'}(\boldsymbol x')) \\+ \frac1{d!} K_1(m_{c}(\boldsymbol x), m_{c'}(\boldsymbol x')) n_{c}(x_j)\Big)
\end{multline*}
where $n_{c}(x_j)$ is the inward weighted normal to the face opposed to $x_j$ in $\gamma({c})$. We finally get the expression of the gradient of the data attachment term as
\[
\partial_{x_j} \|\mu_{(S^{(0)},\boldsymbol x,\boldsymbol \zeta^{(0)})} - \mu_{\mathcal T^{(1)}}\|^2_{W^*} = 2 \partial_{x_j} \langle \mu_{(S^{(0)},\boldsymbol x,\boldsymbol \zeta^{(0)})}, \mu_{(S^{(0)}, \boldsymbol{{\tilde x}}, \boldsymbol{\zeta})} - \mu_{\mathcal T^{(1)}} \rangle_{W^*}.
\]
evaluated with $\boldsymbol{\tilde x} = \boldsymbol x$.
\end{proof}

\bmhead{Conflict of Interest}
MM owns a founder share of Anatomy Works with the
arrangement being managed by Johns Hopkins University
in accordance with its conflict of interest policies. The
remaining authors declare that the research was conducted
in the absence of any commercial or financial relationships
that could be construed as a potential conflict of interest.

\bmhead{Acknowledgements}
Authors would like to acknowledge the Allen Institute for their support via the data contribution.

This work was supported by the National Institutes of
Health (NIH) 
grants R01EB020062
(MM), R01NS102670 (MM), U19AG033655 (MM), P41-
EB031771 (MM), and R01MH105660 (MM);
the National Science Foundation (NSF) 
16-569 NeuroNex contract 1707298 (MM); and the
Computational Anatomy Science Gateway (MM) as part of
the Extreme Science and Engineering Discovery Environment
(XSEDE Towns et al., 2014), which is supported by
the NSF grant ACI1548562,  and the Kavli
Neuroscience Discovery Institute 
supported
by the Kavli Foundation 
(MM).
\bibliography{refs}

\end{document}